\documentclass[10pt, a4paper]{amsart}

\usepackage{amsmath, amssymb, amsfonts, amsthm}
\usepackage[margin=2.0cm]{geometry}
\usepackage[usenames,dvipsnames]{color}
\usepackage[unicode]{hyperref} 
\usepackage{enumerate}
\usepackage[utf8]{inputenc}
\usepackage[T1]{fontenc}

\theoremstyle{plain}
\newtheorem{theorem}{Theorem}[section]
\newtheorem{proposition}[theorem]{Proposition}
\newtheorem{lemma}[theorem]{Lemma}
\newtheorem{corollary}[theorem]{Corollary}
\theoremstyle{definition}
\newtheorem{example}[theorem]{Example}
\newtheorem{definition}[theorem]{Definition}

\theoremstyle{remark}
\newtheorem{remark}[theorem]{Remark}

\newcommand\bfGamma{\mathbf{\Gamma}}

\newcommand\bHen{\mathbf{H}^{e\prime}}
\newcommand\bHepi{\mathbf{H}^{e,\varpi}}
\newcommand\bHepiz{\mathbf{H}^{e,\varpi}_{0}}

\newcommand\Hen{\mathrm{H}^{e\prime}}
\newcommand\Hepi{\mathrm{H}^{e,\varpi}}
\newcommand\Hepiz{\mathrm{H}^{e,\varpi}_{0}}
\newcommand\HepizE{\mathrm{H}^{e,\varpi}_{0,\exists}}

\newcommand\F{\mathbf{F}}
\newcommand\rmF{\mathrm{F}}

\newcommand\VF{\mathbf{VF}}
\newcommand\rmVF{\mathrm{VF}}

\newcommand\PAC{\mathbf{PAC}}
\newcommand\rmPAC{\mathrm{PAC}}

\renewcommand\int{\mathbf{int}}

\newcommand{\mred}{\leq_{\rm m}}

\newcommand{\meq}{\simeq_{\rm m}}
\newcommand{\Tred}{\leq_{\rm T}}

\newcommand{\Teq}{\simeq_{\rm T}}

\newcommand{\bridge}{.}
\newcommand{\arch}{|}

\newcommand\res{\mathrm{res}}

\newcommand{\verum}{\top}
\newcommand{\falsum}{\perp}

\newcommand\axsurj{\hyperref[axsurj]{{\rm(sur)}}}
\newcommand\axmono{\hyperref[axmono]{{\rm(mon)}}}
\newcommand\axwmono{\hyperref[axwmono]{{\rm(wm)}}}
\newcommand\Rfour{\hyperref[R4]{{\rm(R4)}}}
\newcommand\RfourF[1]{\hyperref[R4_F]{{\rm(R4)$_{#1}$}}}

\numberwithin{equation}{section}

\title[]{{\large Interpretations of syntactic fragments of theories of fields}}
\author{Sylvy Anscombe and Arno Fehm}
\address{Universit\'{e} Paris Cit\'{e} and Sorbonne Universit\'{e}, CNRS, IMJ-PRG, F-75013 Paris, France}
\email{sylvy.anscombe@imj-prg.fr}
\address{Institut f\"{u}r Algebra, Technische Universit\"{a}t Dresden, 01062 Dresden, Germany}
\email{arno.fehm@tu-dresden.de}

\begin{document}
 \begin{abstract}
We set up general machinery to study interpretations of fragments of theories.
We then apply this to existential fragments of theories of fields,
and especially of henselian valued fields.
As an application we prove many-one reductions
between various existential theories of fields. 
In particular we exhibit several theories of fields many-one equivalent to the existential theory of $\mathbb{Q}$.
 \end{abstract}
\maketitle

\section{Introduction}

\noindent
The notion of an {\em interpretation} of one structure $N$ in a language $\mathfrak{L}_1$ in another structure $M$ in a language $\mathfrak{L}_2$ is well established in contemporary model theory. 
Usually this refers to an isomorphism between $N$ and a definable quotient inside $M$, cf.~\cite{Hodges_long,Hodges,Poizat,Marker}.
For example, every field interprets every linear algebraic group over it (even definably so, i.e.~without the need to pass to a quotient), and every valued field interprets its value group and its residue field. 
The main importance of this notion of interpretation derives from the fact that many properties (like stability, or decidability) are transferred from the interpreting structure $M$ to the interpreted structure $N$.
The purpose of this work
is to
study interpretations of {\em theories} (rather than structures) with a focus on fragments of languages, for example existential fragments,
with our principal aim to transfer decidability statements.

The first point to note here is that there seems to be less agreement in the literature on what an interpretation of an $\mathfrak{L}_1$-theory $T_1$ in an $\mathfrak{L}_2$-theory $T_2$ is.
Older definitions seem to be of little use nowadays, as Hodges remarks in his discussion in \cite[\S5.4]{Hodges_long}.
One possible notion relates to that of a {\em uniform interpretation} $\Gamma$ of the models of $T_1$ in the models of $T_2$,
in the sense that 
there is one $\mathfrak{L}_2$-formula $\delta_\Gamma$ and
for every unnested atomic $\mathfrak{L}_1$-formula $\varphi$
a corresponding $\mathfrak{L}_2$-formula $\varphi_\Gamma$
such that for every $M\models T_2$ there exists $N\models T_1$
and a surjective map 
$f_\Gamma:\delta_\Gamma(M)\rightarrow N$ such that 
$$
 N\models\varphi(f_\Gamma(a_1),\dots,f_\Gamma(a_m))
  \;\Longleftrightarrow\;
  M\models\varphi_\Gamma(a_1,\dots,a_m),\quad a_1,\dots,a_m\in \delta_\Gamma(M),
$$
so that $f_\Gamma$ induces an $\mathfrak{L}_{1}$-isomorphism $\delta_\Gamma(M)/E_\Gamma\cong N$
(where $E_\Gamma$ is the equivalence relation defined 
by the $\mathfrak{L}_2$-formula assigned to the $\mathfrak{L}_1$-formula $x_1=x_2$),
and thus an interpretation of $N$ in $M$.
Many variants of this are discussed in the above mentioned \cite[\S5.4]{Hodges_long}.

Such a uniform interpretation yields in particular the following two things:
A class function $\sigma:{\rm Mod}(T_2)\rightarrow{\rm Mod}(T_1)$ assigning to every model $M$ of $T_2$ the model $\sigma M=\delta_\Gamma(M)/E_\Gamma$ of $T_1$ that the given formulas produce,
and a translation map 
$\iota:{\rm Sent}(\mathfrak{L}_1)\rightarrow{\rm Sent}(\mathfrak{L}_2)$
from the set of $\mathfrak{L}_1$-sentences to
the set of $\mathfrak{L}_2$-sentences
which are related via
\begin{equation}\label{eqn:sigmaiota}
 \sigma M\models\varphi \;\Longleftrightarrow\;  M\models\iota\varphi 
\end{equation}
for every $M\models T_2$ and every $\mathfrak{L}_1$-sentence $\varphi$.
If the structure map $\sigma$ is surjective (say up to isomorphism, or just up to elementary equivalence), this then gives in particular an interpretation of the theory $T_1$ in the theory $T_2$ in the sense that
\begin{equation}\label{eqn:T2T1iota}
    T_1\models\varphi \;\Longleftrightarrow\; T_2\models\iota\varphi. 
\end{equation}
In particular, one obtains (assuming moreover that the translation map $\iota$ is computable)
a Turing reduction, and more precisely even a many-one reduction from $T_1$ to $T_2$,
and in fact a whole family of reductions 
of $\mathfrak{L}_1$-theories extending $T_1$ to corresponding $\mathfrak{L}_2$-theories extending $T_2$.
We remark that there are important examples where a situation as in (\ref{eqn:sigmaiota}) arises without the existence of an actual interpretation of structures, like when $\sigma$ assigns to a field the inverse system of its absolute Galois group.

What this work does a bit more precisely is to study interpretations of theories with respect to a fragment, 
with the goal of obtaining translation maps as in (\ref{eqn:T2T1iota}) but now with $\varphi$ and $\iota\varphi$ in the relevant fragments. 
For example, when the formulas $\delta_\Gamma$, $\varphi_\Gamma$ in a uniform interpretation $\Gamma$ are all say quantifier-free, then the translation map $\iota$ will map existential sentences to existential sentences, thereby giving (many-one) reductions between existential theories,
see for example \cite{Pasten} for foundational work on interpretations in the context of existential theories.
Moreover, one obtains a monotonicity statement for in this case the existential fragment: 
\begin{equation}\label{eqn:mon}
M,M'\models T_2,{\rm Th}_\exists(M)\subseteq{\rm Th}_\exists(M')\;\Longrightarrow\; {\rm Th}_\exists(\sigma M)\subseteq{\rm Th}_\exists(\sigma M').
\end{equation}
We set up a general formalism to deal with such interpretations of fragments, which for applicability still works with a structure map $\sigma$ but accommodates for the fact that when requiring (\ref{eqn:sigmaiota}) only for certain sentences (like existential sentences), asking for an actual interpretation of the whole structure $\sigma M$ in $M$ is certainly way more than necessary, and in fact there are many more settings where a translation $\iota$ between fragments of theories arises without being induced by an interpretation of structures. 

One important such setting is that of valued fields, where as mentioned above every valued field interprets its residue field (in the sense of a uniform intepretation with $T_1$ the theory of fields and $T_2$ the theory of valued fields), but in general the residue field does not interpret the valued field; however, at least in the context of equicharacteristic henselian valued fields, the existential theory of the residue field determines the existential theory of the valued field (this is the main result of \cite{AF16}),
and in fact monotonely so in the sense of (\ref{eqn:mon}) with the arrow reversed,
which in turn gives an interpretation of the existential theory of the valued field in the existential theory of the residue field in the sense of (\ref{eqn:sigmaiota}) and (\ref{eqn:T2T1iota}),
with a translation map $\iota$ going in the other direction
(which is what we will call an elimination map).
So together one obtains something like a mutual interpretation of the existential theory of the valued field with the existential theory of the residue field, without the presence of a mutual interpretation of structures,
and our formalism is designed to deal with such situations.

After setting up this formalism in Section \ref{sec:bridges},
in Section \ref{section:existential_theories_of_fields} we apply it to existential theories of pseudo-algebraically closed fields and henselian valued fields.
 One of the applications to henselian valued fields that we obtain, using and extending \cite{AF16} and \cite{ADF22}, is the following
 (proven in Section \ref{proof:1.1}): 

\begin{theorem}\label{thm:introhens}
Let $R$ be any theory of fields in the language of rings.
\begin{enumerate}[$(a)$]
\item The set of existential sentences in the language of rings that hold in every model of $R$ is many-one equivalent
to the set of existential sentences in the language of valued fields that hold in every equicharacteristic henselian non-trivially valued field with residue field a model of $R$.

\item The same as (a) applies with "nontrivially valued field" replaced by "valued field with discrete value group and distinguished uniformizer", if consequence \Rfour\ of resolution of singularities from \cite{ADF22} holds.

\item Moreover, the same as (a) applies for every $n$ with
"existential sentence" both times replaced by
"finite conjunction of
existential prenex sentences each of which has at most $n$ existential quantifiers".
\end{enumerate}
\end{theorem}

In the following sections \ref{section:algorithms}-\ref{sec:final} we 
deal with interpretations in the context of
limits of theories, in particular the well-known phenomenon of viewing characteristic zero theories as limits of positive characteristic theories.
As a concrete example we obtain the following result
(proven in Section~\ref{proof:1.2}),
which lists various theories many-one equivalent to the existential theory of $\mathbb{Q}$, whose decidability is a famous open question also known under the name Hilbert's tenth problem for $\mathbb{Q}$, 
see for example \cite{Koenigsmann_survey, Poonen_survey}:

\begin{theorem}\label{thm:introQ}
The following theories are 
many-one equivalent:
\begin{enumerate}[$(a)$]
    \item The existential theory of $\mathbb{Q}$ in the language of rings.
    \item The existential theory of $\mathbb{Q}(\!(t)\!)$ in the language of rings.
    \item The existential theory of $\mathbb{Q}(\!(t)\!)$ in the language of valued fields.
    \item The existential theory of $\mathbb{Q}(\!(t)\!)$ in the language of valued fields with constant $t$.
    \item The existential theory of large fields of characteristic zero in the   
    language of rings.
    \item  The existential theory of large fields in the language of rings.
    \item The existential theory of fields in the language of rings.
\end{enumerate}
\end{theorem}

Finally,
as another application,
we discuss the existential theory of global fields.
Regarding decidability we obtain for example the following
(proven in Section~\ref{proof:1.3}):

\begin{theorem}\label{thm:intro3}
The following are equivalent:
\begin{enumerate}[$(a)$]
    \item The existential theory of infinite fields in the language of rings is decidable.
    \item 
    The existential theory of $\mathbb{Q}$ in the language of rings
    is decidable and
    the existential theory of $\mathbb{F}_p(t)$ in the language of rings is uniformly decidable,
    i.e.~there exists an algorithm that decides given $p\in\mathbb{P}$ and an existential sentence $\varphi$  whether $\varphi$ holds in $\mathbb{F}_p(t)$.
\end{enumerate}
\end{theorem}

We will further apply this formalism
to universal-existential theories of fields,
in particular Laurent series fields and function fields,
in the upcoming work \cite{AF24}.
Another application can be found in \cite{DF}.

\section{Contexts, bridges and arches}
\label{sec:bridges}

\noindent
In this section we set up our formalism in a completely general setting.
Concrete examples and applications are provided in Section \ref{section:existential_theories_of_fields}.

\subsection{Interpretations and eliminations}

We start by introducing the notions of fragments and contexts, which we combine into bridges --- the setting in which we study interpretations and eliminations.

\begin{definition}
A {\em language} $\mathfrak{L}$
is a (possibly multi-sorted) first-order language.
We denote by
${\rm Form}(\mathfrak{L})$ the set of $\mathfrak{L}$-formulas, and by
${\rm Sent}(\mathfrak{L})$ the set of $\mathfrak{L}$-sentences,
where we follow the convention
that $\verum,\falsum\in{\rm Sent}(\mathfrak{L})$
for every language $\mathfrak{L}$.
As usual we call a subset $T\subseteq{\rm Sent}(\mathfrak{L})$ an {\em $\mathfrak{L}$-theory}.
We moreover call a subset $L\subseteq{\rm Sent}(\mathfrak{L})$
which 
contains $\verum$ and $\falsum$
and is closed under $\wedge$ and $\vee$
an {\em $\mathfrak{L}$-fragment}.
The class of models of a theory $T$ is denoted ${\rm Mod}(T)$.
For an $\mathfrak{L}$-structure $M$, we 
denote as usual by ${\rm Th}(M)$ the
$\mathfrak{L}$-theory of $M$, and we
write
${\rm Th}_L(M)={\rm Th}(M)\cap L$
for any $\mathfrak{L}$-fragment $L$.
For an $\mathfrak{L}$-theory $T$ we denote by $T^\vdash$ its deductive closure, and write 
$T_L=T^\vdash\cap L$ for any $\mathfrak{L}$-fragment $L$.
\end{definition}

\begin{remark}\label{rem:relative_complete_1}
We will not make use of this but remark that if $L$ is an $\mathfrak{L}$-fragment and $T\subseteq L$ a consistent theory that is "relatively deductively closed" in the sense that $T=T_L$,
then $T$ is "relatively complete" in the sense that 
$T={\rm Th}_L(M)$ for some $\mathfrak{L}$-structure $M$ if and only if 
whenever $\varphi,\psi\in L$ with $\varphi\vee\psi\in T$, then $\varphi\in T$ or $\psi \in T$.
\end{remark}

\begin{remark}\label{rem:presentation}
For a countable language $\mathfrak{L}$,
when discussing the computability of an $\mathfrak{L}$-theory $T$, we will 
always assume that $\mathfrak{L}$ is {\em presented}
in the sense that it comes with a fixed
injection mapping the symbols of $\mathfrak{L}$ to $\mathbb{N}$
(the precise form of which is irrelevant
and will not be given explicitly when $\mathfrak{L}$ is finite).
By a standard Gödel coding, this induces an injection  $\alpha:\mathrm{Form}(\mathfrak{L})\rightarrow\mathbb{N}$,
and we always assume that $\alpha({\rm Form}(\mathfrak{L}))$ is computable.
A set of formulas $T\subseteq{\rm Form}(\mathfrak{L})$
is then {\em computable} (or {\em decidable}) respectively {\em computably enumerable} if $\alpha(T)\subseteq\mathbb{N}$ is.

If $\mathfrak{L}_1$ and $\mathfrak{L}_2$
are two countable languages with corresponding
injections
$\alpha_i:{\rm Form}(\mathfrak{L}_i)\rightarrow\mathbb{N}$,
and $T_i\subseteq{\rm Form}(\mathfrak{L}_i)$,
we say that a map
$f:T_1\rightarrow T_2$ is {\em computable}
if the induced function $\alpha_1(T_1)\rightarrow\alpha_2(T_2)$ is.
Similarly, we say that $T_1$ is 
{\em many-one reducible}
to $T_2$,
and write $T_1\mred T_2$,
if $\alpha_1(T_1)$ is many-one reducible to $\alpha_2(T_2)$,
i.e.~there exists a computable function $f:\mathbb{N}\rightarrow\mathbb{N}$ such that
$n\in\alpha_1(T_1)$ if and only if
$f(n)\in\alpha_2(T_2)$,
see \cite[Definition 1.6.8]{Soare}.
We say that $T_1$ and $T_2$ are {\em many-one equivalent},
and write $T_1\meq T_2$,
if $T_1\mred T_2$ and $T_2\mred T_1$.
We say that $T_1$ is {\em Turing reducible} to $T_2$, and write $T_1\Tred T_2$,
if $\alpha_1(T_1)$ is Turing reducible to $\alpha_2(T_2)$, 
i.e.~there exists a Turing machine that decides $\alpha_1(T_1)$
with an oracle for $\alpha_2(T_2)$,
see \cite{Turing} and \cite[\S11]{Post}.
We say that $T_1$ and $T_2$ are 
{\em Turing equivalent}, and write $T_1\Teq T_2$, if $T_1\Tred T_2$ and $T_2\Tred T_1$.
When we write $T_1\oplus T_2$ on one of the sides of $\Tred$ or $\mred$ we mean some standard coding of the disjoint union,
e.g.~the set $2\alpha_1(T_1)\cup(2\alpha_2(T_2)+1)\subseteq\mathbb{N}$.
\end{remark}

For the rest of this section let $\mathfrak{L},\mathfrak{L}_1,\mathfrak{L}_2$ be languages.

\begin{definition}\label{def:context_bridge}
A pair $(L,T)$ consisting of an $\mathfrak{L}$-fragment $L$ and an $\mathfrak{L}$-theory $T$ will be called an {\em $\mathfrak{L}$-context}.
We call an $\mathfrak{L}$-context $(L_1,T_1)$ a {\em subcontext} of the $\mathfrak{L}$-context $(L_2,T_2)$
if $L_1\subseteq L_2$ and $T_1\subseteq T_2$.
If $C_1=(L_1,T_1)$ is an $\mathfrak{L}_1$-context and
 $C_2=(L_2,T_2)$ an $\mathfrak{L}_2$-context,
 a {\em translation}
 from $C_1$ to $C_2$ is a map $\tau:L_1\rightarrow L_2$
 with
 $$
   T_1\models\varphi \quad\Longleftrightarrow\quad T_2\models\tau\varphi \quad \mbox{ for all }\varphi\in L_1.
 $$
 A {\em bitranslation} from $C_1$ to $C_2$ is a pair
 $(\tau_1,\tau_2)$ where $\tau_1$ is a translation from $C_1$ to $C_2$
 and $\tau_2$ is a translation from $C_2$ to $C_1$ such that
 \begin{eqnarray}
 \label{eqn:bitrans1}     T_1\models\varphi\leftrightarrow\tau_2\tau_1\varphi & \mbox{ for all }\varphi\in L_1,\\
  \label{eqn:bitrans2} T_2\models\psi\leftrightarrow\tau_1\tau_2\psi &\mbox{ for all }\psi\in L_2.
 \end{eqnarray}
\end{definition}

The following lemma is immediate from the definitions:

\begin{lemma}\label{lem:trans_manyone}
Let $\tau$ be a translation from the context $(L_1,T_1)$ to the context $(L_2,T_2)$. 
Then $(T_1)_{L_1}=\tau^{-1}((T_2)_{L_2})$.
If moreover $L_1$, $L_2$ and $\tau$ are computable,
then $(T_1)_{L_1}\mred (T_2)_{L_2}$.
\end{lemma}

\begin{remark}\label{rem:LT}
For a context $C=(L,T)$,
let $[\varphi]_{C}$ denote the equivalence class of $\varphi\in L$ under the equivalence relation $\sim_{C}$ defined on $L$ by $\varphi\sim_{C}\varphi'$ $:\Leftrightarrow$ $T\models\varphi\leftrightarrow\varphi'$.
Denote by $\mathrm{LT}(C)=\{[\varphi]_{C}\mid\varphi\in L\}$
the quotient of $L$ by $\sim_{C}$,
which we view as a lattice isomorphic
to the sublattice of 
the Lindenbaum--Tarski algebra of $T$ generated by the fragment $L$,
cf.~\cite[Ch.~1, \S.1, Example D, p.~5]{Sikorsky}.

Given contexts $C_{1},C_{2}$,
a function $\tau:L_{1}\rightarrow L_{2}$ is a translation from $C_{1}$ to $C_{2}$
if and only if
$\tau^{-1}([\verum]_{C_{2}})=[\verum]_{C_{1}}$.
A pair $(\tau_{1},\tau_{2})$ of functions $\tau_{1}:L_{1}\rightarrow L_{2}$, $\tau_{2}:L_{2}\rightarrow L_{1}$
satisfies
\eqref{eqn:bitrans1}
if and only if
$\tau_{2}\circ\tau_{1}$
fixes each $\sim_{C_{1}}$-class setwise,
i.e.~$\varphi\sim_{C_{1}}\tau_{2}\tau_{1}\varphi$ for all $\varphi\in L_{1}$.
In particular, there exist pairs $(\tau_{1},\tau_{2})$ which are not bitranslations but both $\tau_{i}$ are translations.
Also there exist pairs $(\tau_{1},\tau_{2})$ satisfying both~\eqref{eqn:bitrans1} and~\eqref{eqn:bitrans2} but which are not bitranslations.
\end{remark}

\begin{definition}
A {\em bridge} is a triple 
$B=(C_1,C_2,\sigma)$
where each $C_i=(L_i,T_i)$ is an $\mathfrak{L}_i$-context
and $\sigma:{\rm Mod}(T_2)\rightarrow{\rm Mod}(T_1)$ is a class function.
We write $C_{1}\bridge C_{2}$ for $B$ when $\sigma$ is clear from the context.
An {\em interpretation}
for the bridge $B$
is a map $\iota:L_1\rightarrow L_2$ such that
\begin{equation}\label{eqn:iota}
 \sigma M\models\varphi \quad\Longleftrightarrow\quad M\models\iota\varphi\quad\mbox{ for all }\varphi\in L_1,M\models T_2,
\end{equation}
and an {\em elimination}
for $B$ is a map
$\epsilon:L_2\rightarrow L_1$ 
such that
\begin{equation}\label{eqn:epsilon}
 \sigma M\models\epsilon\psi
 \quad\Longleftrightarrow\quad 
 M\models\psi 
 \quad\mbox{ for all }\psi\in L_2,M\models T_2.
\end{equation}
A bridge $B'=(C_{1}',C_{2}',\sigma')$, with $C_{i}'=(L_{i}',T_{i}')$, {\em extends} $B$ if $C_{i}$ is a subcontext of $C_{i}'$, for $i=1,2$, and 
$\sigma|_{\mathrm{Mod}(T_{2}')}=\sigma'$.
We denote this by $B\sqsubseteq B'$.
Such an extension is a
{\em fragment extension} if $T_{i}'=T_{i}$, for $i=1,2$;
or it is a
{\em theory extension} if $L_{i}'=L_{i}$, for $i=1,2$.
\end{definition}

\begin{example}\label{ex:main}
Our main example, which we will introduce in all detail in Section~\ref{section:existential_theories_of_fields}, is that of valued fields:
Here $\mathfrak{L}_1$ is the language of rings,
$\mathfrak{L}_2$ is some language of valued fields,
$T_2$ some theory of valued fields whose residue fields are  models of the theory $T_1$,
and $\sigma$ maps a valued field to its residue field.
As there is a uniform interpretation of the residue field in the valued field, if $L_1$ and $L_2$ are compatible fragments
(like all $\mathfrak{L}_i$-sentences, or all existential $\mathfrak{L}_i$-sentences), 
the bridge $((L_1,T_1),(L_2,T_2),\sigma)$ has an interpretation.
If one uses the usual three-sorted language of valued fields,
this interpretation is just relativizing to the residue field sort.
\end{example}

For of the rest of this section 
let 
$C_{i}=(L_i,T_i)$, for $i=1,2$, be an $\mathfrak{L}_i$-context,
and let $B=(C_1,C_2,\sigma)$ be a bridge.

\begin{remark}
A function $\iota:L_{1}\rightarrow L_{2}$ is an interpretation for $B$ if and only if
it is a translation
from
$(L_{1},\mathrm{Th}_{L_{1}}(\sigma M))$
to
$(L_{2},\mathrm{Th}_{L_{2}}(M))$,
for each $M\models T_{2}$.
A function $\varepsilon:L_{2}\rightarrow L_{1}$ is an elimination for $B$
if and only if
it is a translation
from
$(L_{2},\mathrm{Th}_{L_{2}}(M))$
to
$(L_{1},\mathrm{Th}_{L_{1}}(\sigma M))$,
for each $M\models T_{2}$.
\end{remark}

\begin{remark}\label{rem:finding_sigma}
Let $\tau$ be a translation from $C_{1}$ to $C_{2}$ such that
$T_{1}\models\varphi\leftrightarrow\varphi'\implies T_{2}\models\tau\varphi\leftrightarrow\tau\varphi'$,
so that $\tau$ induces a map $\mathrm{LT}(C_{1})\rightarrow\mathrm{LT}(C_{2})$.
One can show that if this is a lattice homomorphism, then there exists a class function $\tilde{\sigma}:\mathrm{Mod}(T_{2})\rightarrow\mathrm{Mod}(T_{1})$ such that
$\tau$ is an interpretation for the bridge
$(C_{1},C_{2},\tilde{\sigma})$.
\end{remark}

The following two lemmas are immediate from the definitions:

\begin{lemma}\label{lem:goes_down}
For $i=1,2$ suppose that $C_{i}$ is a subcontext of an $\mathfrak{L}_i$-context $\hat{C}_{i}=(\hat{L}_i,T_i)$.
Then also $\hat{B}=(\hat{C}_1,\hat{C}_2,\sigma)$
is a bridge,
it is a fragment extension of $B$,
and the following holds:
\begin{enumerate}[$(a)$]
    \item If $\iota$ is an interpretation for $\hat{B}$
    and $\iota(L_{1})\subseteq L_{2}$,
    then
    $\iota|_{L_1}$ is an interpretation for
    $B$.
    \item If $\epsilon$ is an elimination for $\hat{B}$
    and $\epsilon(L_{2})\subseteq L_{1}$,
    then $\epsilon|_{L_2}$ is an elimination for $B$.
\end{enumerate}
\end{lemma}

\begin{lemma}\label{lem:goes_up}
For $i=1,2$ suppose that $C_{i}$ is a subcontext of 
an $\mathfrak{L}_i$-context $\check{C}_{i}=(L_i,\check{T}_i)$
such that $\sigma(\mathrm{Mod}(\check{T}_{2}))\subseteq\mathrm{Mod}(\check{T}_{1})$.
Then also
$\check{B}=(\check{C}_{1},\check{C}_{2},\check{\sigma})$ is a bridge,
where $\check{\sigma}=\sigma|_{\mathrm{Mod}(\check{T}_{2})}$,
and it is a theory extension of $B$.
Moreover, the following holds:
\begin{enumerate}[$(a)$]
\item If $\iota$ is an interpretation for $B$, then $\iota$ is an interpretation
for $\check{B}$.
\item If $\epsilon$ is an elimination for $B$, then $\epsilon$ is an elimination
for $\check{B}$.
\end{enumerate}
\end{lemma}

\begin{definition}
We consider the following property
of the bridge $B$:
\begin{enumerate}
	\item[{\bf(sur)}]\label{axsurj}
	For all $N\models T_{1}$ there exists $M\models T_{2}$ such that $\mathrm{Th}_{L_{1}}(N)=\mathrm{Th}_{L_{1}}(\sigma M)$.
\end{enumerate}
\end{definition}

\begin{lemma}\label{lem:coarsening}
\begin{enumerate}[$(a)$]
    \item For every interpretation $\iota$ for $B$ we have
    \begin{align}\label{eqn:T1T2}
    T_{1}\models\varphi
    \quad\Longrightarrow\quad
    T_{2}\models\iota\varphi,\quad
    \mbox{ for all }\varphi\in L_{1},
    \end{align}
    and if $B$ satisfies \axsurj, then also
    \begin{align}\label{eqn:T1T2b}
    T_{1}\models\varphi
    \quad\Longleftarrow\quad
    T_{2}\models\iota\varphi,\quad
    \mbox{ for all }\varphi\in L_{1},
    \end{align}    
    and so then $\iota$ is a translation from $C_1$ to $C_2$.
    \item For every elimination $\epsilon$  for $B$ we have
    \begin{align}\label{eqn:T2T1}
    T_{1}\models\epsilon\psi
    \quad\Longrightarrow\quad
    T_{2}\models\psi,\quad
    \mbox{ for all }\psi\in L_{2},
    \end{align}
 and if $B$ satisfies \axsurj, then also      \begin{align}\label{eqn:T2T1b}
    T_{1}\models\epsilon\psi
    \quad\Longleftarrow\quad
    T_{2}\models\psi,\quad
    \mbox{ for all }\psi\in L_{2},
    \end{align}  
    and so then $\epsilon$ is a translation from $C_2$ to $C_1$.
\end{enumerate}
\end{lemma}

\begin{proof}
In both (a) and (b)
 the implication $\Longrightarrow$ is
trivial from the definition.
For $\Longleftarrow$ in (a)
let $\varphi\in L_1$ with $T_{2}\models\iota\varphi$.
For every $N\models T_{1}$,
there exists $M\models T_{2}$ with $\mathrm{Th}_{L_{1}}(N)=\mathrm{Th}_{L_{1}}(\sigma M)$
by \axsurj,
thus
$M\models\iota\varphi$, hence $\sigma M\models\varphi$,
and therefore $N\models\varphi$.
The implication $\Longleftarrow$ for (b) is proven similarly.
\end{proof}

\begin{proposition}
\label{rem:2}
If the bridge $B$ has both an interpretation $\iota$ and an elimination $\epsilon$, then
$(\iota,\epsilon)$
satisfies part (\ref{eqn:bitrans2})
of the definition a bitranslation (but $\iota$ and $\epsilon$ need not be translations between $C_1$ and $C_2$).
If in addition $B$ satisfies \axsurj, then
$(\iota,\epsilon)$ is a bitranslation between $C_1$ and $C_2$.
\end{proposition}

\begin{proof}
For $\psi\in L_2$,
by (\ref{eqn:iota}) and (\ref{eqn:epsilon}) we have 
for every $M\models T_2$
that
$M\models\psi$ if and only if $M\models\iota\epsilon\psi$,
which shows (\ref{eqn:bitrans2}).
If $B$ satisfies \axsurj, then Lemma \ref{lem:coarsening}(a) and (b) give that $\iota$, respectively $\epsilon$, is a translation, 
and for $\varphi\in L_{1}$ and $N\models T_{1}$ we have
$$
N\models\varphi
\quad\Longleftrightarrow\quad
\sigma M\models\varphi
\quad\Longleftrightarrow\quad
	M\models\iota\varphi
\quad\Longleftrightarrow\quad
	\sigma M\models\epsilon\iota\varphi
\quad\Longleftrightarrow\quad
	N\models\epsilon\iota\varphi,
 $$
where $M$ is some model of $T_2$ with
$\mathrm{Th}_{L_{1}}(N)=\mathrm{Th}_{L_{1}}(\sigma M)$,
whose existence is guaranteed by \axsurj.
\end{proof}

\begin{definition}
We consider the following monotonicity property
of the bridge $B$:
\begin{enumerate}
	\item[{\bf(mon)}]\label{axmono}
	For all $M,M'\models T_{2}$
	with $\mathrm{Th}_{L_{1}}(\sigma M)\subseteq\mathrm{Th}_{L_{1}}(\sigma M')$
	we have $\mathrm{Th}_{L_{2}}(M)\subseteq\mathrm{Th}_{L_{2}}(M')$.	
\end{enumerate}
\end{definition}

\begin{lemma}\label{lem:boolean}
Suppose $\iota$ is an interpretation for $B$.
Then for every $\varphi,\varphi'\in L_{1}$,
$T_{2}\models\iota(\varphi\wedge\varphi')\leftrightarrow(\iota\varphi\wedge\iota\varphi')$
and 
$T_{2}\models\iota(\varphi\vee\varphi')\leftrightarrow(\iota\varphi\vee\iota\varphi')$.
\end{lemma}
\begin{proof}
This is very easy to check.
\end{proof}

\begin{proposition}\label{prp:elimination}
If $B$ admits an elimination, then it satisfies \axmono.
Conversely,
if $B$ admits an interpretation $\iota$ and satisfies \axmono,
then it admits an elimination.
In that case, if in addition
$L_1$ and $L_2$ are computable,
$T_2$ is computably enumerable
and $\iota$ is computable,
then $B$ admits a computable elimination.
\end{proposition}

\begin{proof}
If $\epsilon$ is an elimination for $B$, then
for $M,M'\models T_2$ with ${\rm Th}_{L_1}(\sigma M)\subseteq{\rm Th}_{L_1}(\sigma M')$
we have
$$
 M\models\psi \quad\Longleftrightarrow\quad
 \sigma M\models\epsilon\psi
 \quad\Longrightarrow\quad
 \sigma M'\models\epsilon\psi \quad\Longleftrightarrow\quad
 M'\models \psi
$$
for every $\psi\in L_2$,
hence ${\rm Th}_{L_2}(M)\subseteq{\rm Th}_{L_2}(M')$.

Now suppose that $B$ satisfies \axmono\ and admits an interpretation $\iota$.
We claim that for every $\psi\in L_{2}$ there exists $\varphi\in L_{1}$ with $T_{2}\models\psi\leftrightarrow\iota\varphi$.
Indeed, for every $M,M'\models T_{2}$ with $M\models\psi$ and $M'\not\models\psi$, by \axmono\ there exists $\alpha\in L_{1}$ with $\sigma M\models\alpha$ and $\sigma M'\not\models\alpha$.
Thus $M\models\iota\alpha$ and $M'\not\models\iota\alpha$.
Therefore, the separation lemma \cite[Lemma 3.1.3]{PrestelDelzell}
gives that there exist $\alpha_{ij}\in L_{1}$ such that $T_{2}\models\psi\leftrightarrow\bigwedge_{i}\bigvee_{j}\iota\alpha_{ij}$.
By Lemma~\ref{lem:boolean},
$T_{2}\models\bigwedge_{i}\bigvee_{j}\iota\alpha_{ij}\leftrightarrow\iota(\bigwedge_{i}\bigvee_{j}\alpha_{ij})$.
Thus the choice $\varphi:=\bigwedge_{i}\bigvee_{j}\alpha_{ij}$ satisfies
the claim.
By defining $\epsilon\psi$ to be $\varphi$ for any such $\varphi$, we have
$M\models\psi \Leftrightarrow
 M\models\iota\varphi \Leftrightarrow
 \sigma M\models \epsilon\psi$
for any $M\models T_2$.

Assume now in addition that $L_1$ and $L_2$ are computable,
$T_2$ is computably enumerable and $\iota$ is computable.
Then we can obtain a computable $\epsilon$ as follows:
Fix a computable enumeration
$\varphi_1,\varphi_2,\dots$ of $L_{1}$
and a computable enumeration
$P_1,P_2,\dots$
of the proofs from $T_{2}$
(here we use $\alpha:{\rm Form}(\mathfrak{L}_2)\rightarrow\mathbb{N}$
to obtain an injection from the
set of finite sequences of $\mathfrak{L}_2$-formulas
into $\mathbb{N}$
in the usual way).
The ordering $\prec$ on $\mathbb{N}\times\mathbb{N}$ defined by 
\begin{eqnarray}\label{eqn:ordering}
 (n,m)\prec (n',m')\;\Longleftrightarrow\; n+m<n'+m'\vee (n+m=n'+m'\wedge m<m')    
\end{eqnarray}
is computable and has order type $\omega$.
Given $\psi\in L_{2}$,
by the previous paragraph there exist $n$ and $m$
such that
$P_n$ is a proof of
$T_{2}\vdash \psi\leftrightarrow\iota\varphi_m$.
We take the $\prec$-minimal such pair $(n,m)$
and define $\epsilon\psi:=\varphi_{m}$.
\end{proof}

\subsection{Extensions of bridges}

In certain circumstances a bridge admitting an interpretation may be extended in such a way that the interpretation also extends.
To deal with such situations in full generality, we combine two bridges into an arch.

\begin{definition}\label{def:arch}
An {\em arch} is a triple $A=(B,\hat{B},\iota)$ such that $B\sqsubseteq\hat{B}$ is a fragment extension, $\iota$ is an interpretation for $\hat{B}$,
and $\iota|_{L_1}$ is an interpretation for $B$.
We write $B\arch\hat{B}$ for $A$ when $\iota$ is clear from the context.
\end{definition}

\begin{example}
In our main example of valued fields (see Example \ref{ex:main}),
we use arches to capture the fact that
we have an interpretation $\iota$ defined on arbitrary sentences
although
the eliminations are usually defined only on existential or universal sentences (see, for example, Corollary \ref{cor:henselian}),
but restricting to existential (or universal) sentences throughout
would be too limiting,
as we would like to work 
for example in the class of valued fields with pseudofinite residue field,
which is axiomatized neither by existential nor by universal sentences.
Most of our applications of arches will involve
Corollary \ref{lem:pattern} below.
\end{example}

\begin{lemma}\label{lem:sur_goes_up}
Suppose that $B$ admits an interpretation $\iota$.
Let $R\subseteq L_1$ and write $\check{T}_1=T_1\cup R$, $\check{T}_2=T_2\cup\iota R$,
$\check{\sigma}=\sigma|_{{\rm Mod}(\check{T}_2)}$,
and
$\check{C}_i=(L_i,\check{T}_i)$.
Then
$\check{B}=(\check{C}_1,\check{C}_2,\check{\sigma})$ is a bridge,
it is a theory extension of $B$, and
in particular
$\check{T}_{2}^{\vdash}=(T_{2}\cup\iota(R_{L_{1}}))^{\vdash}=(T_{2}\cup\iota(\check{T}_{1,L_{1}}))^{\vdash}$.
Moreover, if $B$
satisfies \axsurj,
then so does $\check{B}$.
\end{lemma}

\begin{proof}
If $M\models\check{T}_2=T_2\cup\iota R$, then in particular
$\sigma M\models R$, but also $\sigma M\models T_1$ as $B$ is a bridge,
hence
$\sigma M\models \check{T}_1$, and so
$\check{B}$ is a bridge.
In particular, $\sigma M\models\check{T}_{1,L_1}$,
hence $\check{T}_{2}^{\vdash}\supseteq(T_{2}\cup\iota(\check{T}_{1,L_{1}}))^{\vdash}$,
and the inclusions
$\check{T}_{2}^{\vdash}\subseteq(T_{2}\cup\iota(R_{L_{1}}))^{\vdash}\subseteq(T_{2}\cup\iota(\check{T}_{1,L_{1}}))^{\vdash}$
are trivial
as $R\subseteq L_1$.
Now suppose that $B$ satisfies \axsurj,
and let $N\models\check{T}_{1}=T_{1}\cup R$.
By \axsurj\ there exists $M\models T_{2}$ with $\mathrm{Th}_{L_{1}}(N)=\mathrm{Th}_{L_{1}}(\sigma M)$.
Since $R\subseteq L_{1}$, this implies that $\sigma M\models R$, and so $M\models\iota R$.
Therefore $\check{B}$ satisfies \axsurj.
\end{proof}    

For the rest of this section $\hat{B}$ will denote a fragment extension of $B=(C_1,C_2,\sigma)$,
and we write 
$\hat{B}=(\hat{C}_{1},\hat{C}_{2},\sigma)$,
with $\hat{C}_{i}=(\hat{L}_{i},T_{i})$.

\begin{proposition}\label{prop:arch}
Let $A=(B,\hat{B},\iota)$ be an arch where $B$ admits an elimination $\epsilon$.
We suppose that 
$\hat{B}$ satisfies \axsurj.
Then for each $R\subseteq \hat{L}_1$,
$(\iota|_{L_{1}},\epsilon)$ is a bitranslation between
$(L_1,T_1\cup R)$ and $(L_2,T_2\cup\iota R)$.
\end{proposition}

\begin{proof}
Define $\check{T}_1=T_1\cup R$,
$\check{T}_2=T_2\cup \iota R$.
Since $\sigma(\mathrm{Mod}(\check{T}_{2}))\subseteq\mathrm{Mod}(\check{T}_{1})$,
both
$B(R):=((L_{1},\check{T}_{1}),(L_{2},\check{T}_{2}),\check{\sigma})$
and
$\hat{B}(R):=((\hat{L}_{1},\check{T}_{1}),(\hat{L}_{2},\check{T}_{2}),\check{\sigma})$
are bridges,
where $\check{\sigma}=\sigma|_{\mathrm{Mod}(\check{T}_{2})}$
(Lemma~\ref{lem:goes_up}).
Moreover,
$\iota|_{L_{1}}$ is an interpretation for $B(R)$ (Lemmas~\ref{lem:goes_down}(a), \ref{lem:goes_up}(a))
and also
$\epsilon$ is an elimination for $B(R)$
(Lemma~\ref{lem:goes_up}(b)).
Furthermore,
since $\hat{B}$ satisfies \axsurj,
so does $\hat{B}(R)$,
by Lemma~\ref{lem:sur_goes_up},
which trivially implies that also $B(R)$ satisfies \axsurj.   
The claim follows from Proposition \ref{rem:2}.
\end{proof}

\begin{corollary}\label{lem:pattern}
Let $A=(B,\hat{B},\iota)$ be an arch.
We suppose
\begin{enumerate}[$(i)$]
\item\label{pattern.i}
$L_{1},L_{2},\hat{L}_1,\hat{L}_2$ are computable,
$\iota$ is computable, and
$T_{2}$ is computably enumerable%
\footnote{The computability of $L_{1}$ and $\hat{L}_{1}$ is to be understood with respect to the same injection $\alpha_1:\mathrm{Form}(\mathfrak{L}_{1})\rightarrow\mathbb{N}$, as in Remark~\ref{rem:presentation}; and similarly for $L_{2}$, $\hat{L}_{2}$ and $T_2$.},
\item\label{pattern.ii}
$\hat{B}$ satisfies \axsurj, and
\item\label{pattern.iii}
$B$ satisfies \axmono.
\end{enumerate}
Then
\begin{enumerate}[{\rm(I)}] 
\item
$B$ admits a computable elimination.
\item
For each $R\subseteq\hat{L}_{1}$,  
$(T_{1}\cup R)_{L_{1}}\meq(T_{2}\cup\iota R)_{L_{2}}$.
\item
$\mathrm{Th}_{L_{2}}(M)=(T_{2}\cup\iota\mathrm{Th}_{\hat{L}_{1}}(\sigma M))_{L_{2}}=(T_{2}\cup\iota\mathrm{Th}_{L_{1}}(\sigma M))_{L_{2}}$, for each $M\models T_{2}$.
\end{enumerate}
\end{corollary}

\begin{proof}
Since $B$ satisfies \axmono\
and admits an interpretation,
it admits an elimination $\epsilon$ (Proposition~\ref{prp:elimination}),
which, together with the computability assumptions,
may be taken to be computable.

For $R\subseteq\hat{L}_{1}$,
$(\iota|_{L_{1}},\epsilon)$ is a bitranslation between
$(L_1,T_1\cup R)$ and $(L_2,T_2\cup\iota R)$ by 
Proposition \ref{prop:arch}.
Since both $\iota|_{L_1}$ and $\epsilon$ are computable,
this gives a many-one equivalence by Lemma \ref{lem:trans_manyone}.

For $M\models T_{2}$
the inclusions
$\mathrm{Th}_{L_{2}}(M)\supseteq(T_{2}\cup\iota\mathrm{Th}_{\hat{L}_{1}}(\sigma M))_{L_{2}}\supseteq(T_{2}\cup\iota\mathrm{Th}_{L_{1}}(\sigma M))_{L_{2}}$
are trivial.
Let $M'\models T_{2}\cup\iota\mathrm{Th}_{L_{1}}(\sigma M)$,
so that $\sigma M'\models\mathrm{Th}_{L_{1}}(\sigma M)$.
Then $M'\models\mathrm{Th}_{L_{2}}(M)$,
since $B$ satisfies \axmono,
which proves {\rm(III)}.
\end{proof}

With the preceding corollary, the majority of the machinery needed for our applications in Section~\ref{section:existential_theories_of_fields} is complete.
The following two subsections deal with 
slightly more technical aspects of it,
which will be useful later in Section~\ref{sec:final} when we give our second tranche of applications.
The reader might want to thus skip ahead to the algebraic examples in Section~\ref{section:existential_theories_of_fields}.

\subsection{Families of theories and negations of fragments}

In our applications in Section~\ref{sec:final},
varying the characteristic of fields,
we will be extending bridges not by one theory $R\subseteq\hat{L}_1$ but by a family of theories.
The conclusion (\ref{eqn:Qlemma}) of the following proposition describes some fundamental relationships that one would ideally hope for in such a setting.
As stated, this proposition 
immediately gives a relatively straightforward proof of
Proposition~\ref{cor:F_and_P_combined}{\rm(I)},
and a conditional proof of
Proposition~\ref{cor:F_and_P_combined}{\rm(II)},
as described in Remark~\ref{rem:wm_not_R4}.
Making this proof unconditional is our main motivation for the more technical considerations in
sections~\ref{section:wm} and~\ref{section:stratify_arch},
each of which obtains parts of (\ref{eqn:Qlemma}) under weaker assumptions.

\begin{proposition}\label{prp:intersection}
Let $A=(B,\hat{B},\iota)$ be an arch.
Suppose that $\hat{B}$ satisfies \axsurj{},
and that $B$ 
admits an elimination~$\epsilon$.
Let $I\neq\emptyset$
and $R_{i}\subseteq\hat{L}_{1}$, for $i\in I$.
Then
\begin{equation}
 \label{eqn:Qlemma}   (T_{2}\cup\iota(\bigcap_{i\in I}(T_1\cup R_i)_{{L}_{1}}))_{L_{2}}
 =
 (T_{2}\cup\iota(\bigcap_{i\in I}(T_1\cup R_i)_{\hat{L}_{1}}))_{L_{2}}
 =
 \bigcap_{i\in I}(T_{2}\cup\iota R_{i})_{L_{2}}.
\end{equation}
In particular,
$(T_{2}\cup\iota((T_{1}\cup R)_{L_{1}}))_{L_{2}}
=
(T_{2}\cup\iota((T_{1}\cup R)_{\hat{L}_{1}}))_{L_{2}}
=
(T_{2}\cup\iota R)_{L_{2}}$
for every $R\subseteq\hat{L}_{1}$.
\end{proposition}
\begin{proof}
For $R\subseteq\hat{L}_1$, we write $\bar{R}=(T_1\cup R)_{\hat{L}_1}$.
First of all note that $R\subseteq\bar{R}=\bar{\bar{R}}$,
as
$(T_{1}\cup R)_{\hat{L}_{1}}\subseteq(T_{1}\cup\bar{R})_{\hat{L}_{1}}\subseteq(T_{1}^{\vdash}\cup(T_{1}\cup R)^{\vdash})_{\hat{L}_{1}}\subseteq((T_{1}\cup T_{1}\cup R)^{\vdash})_{\hat{L}_{1}}=(T_{1}\cup R)_{\hat{L}_{1}}$.
In particular
$(T_1\cup \bar{R})_{{L}_1}=(T_1\cup R)_{{L}_1}$.
Moreover, $(T_{2}\cup\iota R)^{\vdash}=(T_{2}\cup\iota\bar{R})^{\vdash}$,
and in particular, $(T_{2}\cup\iota R)_{L_{2}}=(T_{2}\cup\iota\bar{R})_{L_{2}}$,
by Lemma~\ref{lem:sur_goes_up} applied to $\hat{B}$.
Note that by Lemma~\ref{lem:sur_goes_up}, for any $R\subseteq\hat{L}_{1}$,
$((\hat{L}_{1},T_{1}\cup R),(\hat{L}_{2},T_{2}\cup\iota R),\sigma|_{\mathrm{Mod}(T_{2}\cup\iota R)})$ is a bridge satisfying \axsurj.
In particular,
$((L_{1},T_{1}\cup R),(L_{2},T_{2}\cup\iota R),\sigma|_{\mathrm{Mod}(T_{2}\cup\iota R)})$
is a bridge satisfying \axsurj.
By Lemma \ref{lem:goes_up}(b),
$\epsilon$ is an elimination also for this bridge.

By the first paragraph, we may assume that $R_{i}=\bar{R}_{i}$, for each $i\in I$.
Hence we want to prove that
$$
 (T_{2}\cup\iota(\bigcap_{i\in I}(R_{i})_{L_1}))_{L_{2}}
 =
 (T_{2}\cup\iota(\bigcap_{i\in I}R_{i}))_{L_{2}}
 =
 \bigcap_{i\in I}(T_{2}\cup\iota R_{i})_{L_{2}}.
$$ 
Both inclusions
$\subseteq$ are trivial.
It remains to show that
$\bigcap_{i\in I}(T_{2}\cup\iota R_{i})_{L_{2}}\subseteq (T_{2}\cup\iota(\bigcap_{i\in I}(R_{i})_{L_1}))_{L_{2}}$,
so let $\varphi\in\bigcap_{i\in I}(T_{2}\cup\iota R_{i})_{L_{2}}$.
For each $i\in I$, $T_{2}\cup\iota R_{i}\models\varphi$.
Since 
$(({L}_{1},T_{1}\cup R_{i}),({L}_{2},T_{2}\cup\iota R_{i}),\sigma|_{\mathrm{Mod}(T_{2}\cup\iota R_{i})})$
satisfies \axsurj\
and has elimination $\epsilon$
for each $i\in I$, 
we have
$T_{1}\cup R_{i}\models \epsilon\varphi$
by (\ref{eqn:T2T1b}).
Thus
$\epsilon\varphi\in\bigcap_{i\in I}(T_{1}\cup R_{i})_{{L}_{1}}
=\bigcap_{i\in I}(R_{i})_{L_{1}}
\subseteq T_{1}\cup\bigcap_{i\in I}(R_{i})_{L_{1}}$.
Then
$T_{2}\cup\iota\bigcap_{i\in I}(R_{i})_{L_{1}}\models\varphi$
by (\ref{eqn:T2T1}).
\end{proof}

The following lemma abstracts the general principle that statements about existential sentences have natural analogues about universal sentences,
and often even for boolean combinations of them.
For an $\mathfrak{L}$-fragment $L$ write $\neg L=\{\neg\varphi:\varphi\in L\}$ 
and let $\bar{L}$ be the smallest fragment that contains $L\cup\neg L$.
For a context $C=(L,T)$ let $\bar{C}=(\bar{L},T)$,
and for a bridge $B=(C_1,C_2,\sigma)$ let $\bar{B}=(\bar{C}_1,\bar{C}_2,\sigma)$.

\begin{lemma}\label{lem:neg}
If $B$ satisfies \axmono\, then so does $\bar{B}$.
In particular, if $A=(B,\hat{B},\iota)$ is an arch 
that satisfies the assumptions of Corollary \ref{lem:pattern},
with $\neg L_i\subseteq\hat{L}_i$ for $i=1,2$
and $\iota(\neg L_1)\subseteq\neg L_2$,
then also $(\bar{B},\hat{B},\iota)$ is an arch that satisfies these assumptions.
\end{lemma}

\begin{proof}
For $M,M'\models T_2$,
${\rm Th}_{\bar{L}_1}(\sigma M)\subseteq{\rm Th}_{\bar{L}_1}(\sigma M')$
implies 
${\rm Th}_{{L}_1}(\sigma M)\subseteq{\rm Th}_{{L}_1}(\sigma M')$
and
${\rm Th}_{\neg{L}_1}(\sigma M)\subseteq{\rm Th}_{\neg{L}_1}(\sigma M')$,
equivalently
${\rm Th}_{{L}_1}(\sigma M')\subseteq{\rm Th}_{{L}_1}(\sigma M)$,
hence by \axmono\ 
${\rm Th}_{{L}_2}(M)\subseteq{\rm Th}_{{L}_2}(M')$
and 
${\rm Th}_{{L}_2}(M')\subseteq{\rm Th}_{{L}_2}(M)$,
and therefore
${\rm Th}_{\bar{L}_2}(M)={\rm Th}_{\bar{L}_2}(M')$,
in particular
${\rm Th}_{\bar{L}_2}(M)\subseteq{\rm Th}_{\bar{L}_2}(M')$.
\end{proof}

\begin{remark}
Regarding the consequences of  Corollary \ref{lem:pattern} for $(\bar{B},\hat{B},\iota)$,
of course one can easily construct an elimination for $\bar{B}$ from an elimination for $B$,
but it is not clear how to deduce
$(T_1\cup R)_{\bar{L}_1}\meq (T_2\cup\iota R)_{\bar{L}_2}$
directly from
$(T_1\cup R)_{{L}_1}\meq (T_2\cup\iota R)_{{L}_2}$.
\end{remark}

\subsection{Weak monotonicity}
\label{section:wm}
We 
conclude this section with
a weak variant of \axmono\ in the context of
a fragment extension.
In Section \ref{sec:final} 
this will help us to remove the hypothesis \Rfour\
in some places, cf.~Remark \ref{rem:wm_not_R4},
as mentioned above.

\begin{definition}\label{def:wmono}
We consider the following property of the fragment extension $B\sqsubseteq\hat{B}$:
\begin{enumerate}
	\item[{\bf(wm)}]\label{axwmono}
	For all $M\models T_{2}$ and all $N\models T_{1}$,
	if $\mathrm{Th}_{L_{1}}(N)\subseteq\mathrm{Th}_{L_{1}}(\sigma M)$
	then there exists $M'\models T_{2}$ with
	$\mathrm{Th}_{\hat{L}_{1}}(\sigma M')=\mathrm{Th}_{\hat{L}_{1}}(N)$
	and
	$\mathrm{Th}_{L_{2}}(M')\subseteq\mathrm{Th}_{L_{2}}(M)$.
\end{enumerate}
\end{definition}

\begin{lemma}\label{lem:wmono}
	Suppose 
		$\hat{B}$ satisfies \axsurj\ and
		$B$ satisfies \axmono.
	Then 
	$B\sqsubseteq\hat{B}$
	satisfies \axwmono.
\end{lemma}
\begin{proof}
	Let $M\models T_{2}$
	and $N\models T_{1}$ with $\mathrm{Th}_{L_{1}}(N)\subseteq\mathrm{Th}_{L_{1}}(\sigma M)$.
	By \axsurj, there exists $M'\models T_{2}$ with $\mathrm{Th}_{\hat{L}_{1}}(N)=\mathrm{Th}_{\hat{L}_{1}}(\sigma M')$.
	In particular $\mathrm{Th}_{L_{1}}(\sigma M')\subseteq\mathrm{Th}_{L_{1}}(\sigma M)$.
	By \axmono, we have $\mathrm{Th}_{L_{2}}(M')\subseteq\mathrm{Th}_{L_{2}}(M)$.
\end{proof}

\begin{lemma}\label{lem:extension}
Let $T$ be an $\mathfrak{L}$-theory and let $L$ be an $\mathfrak{L}$-fragment.
Suppose that $M\models T_{L}$.
Then there exists $N\models T$ such that $M\models\mathrm{Th}_{L}(N)$.
\end{lemma}
\begin{proof}
It suffices to show that
$T\cup\mathrm{Th}_{(\neg)L}(M)$ is consistent, because any model
$N\models T\cup\mathrm{Th}_{(\neg)L}(M)$ satisfies $M\models \mathrm{Th}_{L}(N)$.
By the compactness theorem, it suffices to show that
$T\cup\mathrm{Th}_{(\neg)L}(M)$
is finitely consistent.
Let
$\neg\varphi_{1},\ldots,\neg\varphi_{n}\in\mathrm{Th}_{(\neg)L}(M)$,
so then
$M\models\neg\bigvee_{i\leq n}\varphi_{i}$.
Since $L$ is a fragment, we have
$\bigvee_{i\leq n}\varphi_{i}\in L$
and so
$\neg\bigvee_{i\leq n}\varphi_{i}\in\mathrm{Th}_{(\neg)L}(M)$.
Therefore
$\bigvee_{i\leq n}\varphi_{i}\notin\mathrm{Th}_{L}(M)$,
and in particular
$T\not\models\bigvee_{i\leq n}\varphi_{i}$,
which proves that $T\cup\mathrm{Th}_{(\neg)L}(M)$ is finitely consistent.
\end{proof}

\begin{lemma}[]\label{lem:incomplete_monotonicity}
Let $A=(B,\hat{B},\iota)$ be an arch such that $B\sqsubseteq\hat{B}$ satisfies \axwmono.
Then
	for all $R,R'\subseteq\hat{L}_1$,
	if $(T_{1}\cup R)_{L_{1}}\subseteq(T_{1}\cup R')_{L_{1}}$ 
	then $(T_{2}\cup\iota R)_{L_{2}}\subseteq(T_{2}\cup\iota R')_{L_{2}}$.
\end{lemma}

\begin{proof}
	Let $M\models T_{2}\cup\iota R'$.
	Then
	$\sigma M\models T_{1}\cup R'$,
	and in particular
	$\sigma M\models(T_{1}\cup R')_{L_{1}}$,
	and so $(T_{1}\cup R)_{L_{1}}\subseteq(T_{1}\cup R')_{L_{1}}\subseteq\mathrm{Th}_{L_{1}}(\sigma M)$.
	By Lemma~\ref{lem:extension},
	there exists $N\models T_{1}\cup R$ such that $\mathrm{Th}_{L_{1}}(N)\subseteq\mathrm{Th}_{L_{1}}(\sigma M)$.
	By \axwmono\
	there exists $M'\models T_{2}$
	with
	$\mathrm{Th}_{\hat{L}_{1}}(\sigma M')=\mathrm{Th}_{\hat{L}_{1}}(N)$
	and $\mathrm{Th}_{L_{2}}(M')\subseteq\mathrm{Th}_{L_{2}}(M)$.
	Thus $\sigma M'\models R$,
	whence $M'\models T_{2}\cup\iota R$.
	Therefore $M\models(T_{2}\cup\iota R)_{L_{2}}$,
	from which it follows that
	$T_{2}\cup\iota R'\models(T_{2}\cup\iota R)_{L_{2}}$.
\end{proof}

\begin{remark}\label{rem:relative_complete_2}
To clarify the relation between \axmono\ and \axwmono, 
we remark that with $B$ and $\hat{B}$ as above and $\hat{B}$ satisfying \axsurj,
one can show that $B$ satisfies \axmono\ if and only if
the pair $\hat{B},B$ satisfies \axwmono\ and in addition
$B$ satisfies the following symmetric weakening of \axmono:
For all $M,M'\models T_{2}$
	with
	$\mathrm{Th}_{L_{1}}(\sigma M)=\mathrm{Th}_{L_{1}}(\sigma M')$
	we have
	$\mathrm{Th}_{L_{2}}(M)=\mathrm{Th}_{L_{2}}(M')$.
\end{remark}

\begin{remark}
Under the assumptions of Corollary~\ref{lem:pattern},
we have that
\begin{enumerate}
\setcounter{enumi}{1}
\item[{\rm(II')}]
For each $\mathfrak{L}_{1}$-theory $R'$, $(T_{1}\cup R')_{L_{1}}\meq(T_{2}\cup\iota((T_{1}\cup R')_{\hat{L}_{1}}))_{L_{2}}$,
\end{enumerate}
which follows from applying Corollary~\ref{lem:pattern}(II) to $R=(T_{1}\cup R')_{\hat{L}_{1}}$.
Similar adaptations can be made to the conclusions of  
Propositions~\ref{prop:arch} and~\ref{prp:intersection}
and Lemma~\ref{lem:incomplete_monotonicity}
to allow $\mathfrak{L}_{1}$-theories.
\end{remark}

\section{Existential theories of fields}
\label{section:existential_theories_of_fields}

\noindent
In this section we discuss a few examples where the general machinery of the previous two sections applies.
Throughout, we are interested in existential fragments,
defined in Section \ref{sec:Efragments}.
In Section \ref{sec:PAC} we first treat theories of pseudo-algebraically closed fields,
and in sections \ref{sec:HeE}-\ref{sec:HeEn}
various theories of henselian valued fields.
In each of these cases we first verify the assumptions of Corollary~\ref{lem:pattern}
and then deduce consequences.

\subsection{Existential fragments}
\label{sec:Efragments}

\begin{definition}\label{def:Efragments}
For a language $\mathfrak{L}$,
an $\mathfrak{L}$-formula $\varphi(x_{1},\ldots,x_{m})$
is {\em existential}
(resp.~{\em universal})
if it is of the form
\begin{align*}
&
    \exists y_{1}\ldots\exists y_{n}\;\psi(x_{1},\ldots,x_{m},y_{1},\ldots,y_{n})\\
\text{\large(resp.~}\quad 
&
    \forall y_{1}\ldots\forall y_{n}\;\psi(x_{1},\ldots,x_{m},y_{1},\ldots,y_{n})
\text{\large)},
\end{align*}
for some $n\geq0$ and an $\mathfrak{L}$-formula $\psi(x_{1},\ldots,x_{m},y_{1},\ldots,y_{n})$ without quantifiers.
Note that the quantifiers appearing in $\varphi$ each range over one of the sorts of $\mathfrak{L}$.
We denote by $\mathrm{Sent}_\exists(\mathfrak{L})$,
respectively  $\mathrm{Sent}_{\exists_{n}}(\mathfrak{L})$,
the smallest $\mathfrak{L}$-fragment
containing all existential $\mathfrak{L}$-sentences,
respectively all existential $\mathfrak{L}$-sentences
with at most $n$ quantifiers.
Similarly, we denote by
$\mathrm{Sent}_\forall(\mathfrak{L})$
and
$\mathrm{Sent}_{\forall_{n}}(\mathfrak{L})$
the sets defined analogously with universal sentences.

For an $\mathfrak{L}$-theory $T$, we write
$T_{\exists}=T_{\mathrm{Sent}_{\exists}(\mathfrak{L})}$,
$T_{\exists_{n}}=T_{\mathrm{Sent}_{\exists_{n}}(\mathfrak{L})}$,
$T_{\forall}=T_{\mathrm{Sent}_{\forall}(\mathfrak{L})}$,
and
$T_{\forall_{n}}=T_{\mathrm{Sent}_{\forall_{n}}(\mathfrak{L})}$.
For an $\mathfrak{L}$-structure $M$ we also
write ${\rm Th}_\exists(M)$ for ${\rm Th}_{\mathrm{Sent}_{\exists}(\mathfrak{L})}(M)$ and so on.
\end{definition}

\begin{remark}
The fragments $\mathrm{Sent}_{\exists_{n}}(\mathfrak{L})$ form a chain, ordered by inclusion, and the union is the fragment $\mathrm{Sent}_{\exists}(\mathfrak{L})$.
Similarly $\mathrm{Sent}_{\forall}(\mathfrak{L})$ is the union of the chain of fragments $\mathrm{Sent}_{\forall_{n}}(\mathfrak{L})$.
By the definition of a fragment, a sentence in ${\rm Sent}_\exists(\mathfrak{L})$ (respectively, ${\rm Sent}_\exists(\mathfrak{L})$)
is obtained from finitely many existential (respectively, universal) $\mathfrak{L}$-sentences by conjunctions and disjunctions.
We remark that our use of $\exists_n$ and $\forall_n$
is very different from the one in \cite{Hodges},
where for example an "$\exists_1$-sentence" is 
an element of what we call ${\rm Sent}_\exists(\mathfrak{L})$.
\end{remark}

\begin{remark}
When deciding the existential fragment of a {\em complete} $\mathfrak{L}$-theory,
i.e.~${\rm Th}_\exists(M)$ for some $\mathfrak{L}$-structure $M$,
by induction on the structure of formulas
it suffices to decide which existential $\mathfrak{L}$-sentences hold in $M$:
$\varphi\wedge\psi\in{\rm Th}_\exists(M)$ if and only if $\varphi\in{\rm Th}_\exists(M)$ and $\psi\in{\rm Th}_\exists(M)$,
and
$\varphi\vee\psi\in{\rm Th}_\exists(M)$ if and only if $\varphi\in{\rm Th}_\exists(M)$ or $\psi\in{\rm Th}_\exists(M)$.
(This does not hold in general for possibly incomplete theories.)
The same goes for checking inclusions like
${\rm Th}_\exists(M)\subseteq{\rm Th}_\exists(M')$ for $\mathfrak{L}$-structures $M,M'$.
\end{remark}

\begin{remark}\label{rem:AF16}
Since we are going to use results from \cite{AF16}
we should point out that the definition of {\em $\exists$-sentence} used there,
namely "logically equivalent
to a formula in prenex normal form with only existential quantifiers" is not suitable 
for the arguments about decidability given there nor for our use here.
In fact, in general it is not possible to decide whether a given sentence is an $\exists$-sentence in that sense:
For example, if $\mathfrak{L}=\{R\}$ with a binary predicate symbol $R$, and $T=\emptyset$, then $T^\vdash$ is undecidable since it encodes graph theory
(which is undecidable, see e.g.~\cite[Corollary 28.5.3]{FJ}), but $T_\exists$ is trivially decidable, 
which implies that it is not possible to decide whether a given sentence $\varphi$ is equivalent 
to some $\psi\in{\rm Sent}_\exists(\mathfrak{L})$,
as such a $\psi$ could be found effectively.
However, this poses no problems for either the results in \cite{AF16} or our application here,
as one can simply replace the definition of $\exists$-sentence by being an element of what we here call the existential fragment ${\rm Sent}_\exists$, and then everything goes through.
\end{remark}

\subsection{PAC fields}
\label{sec:PAC}
For definition and background on pseudo-algebraically closed (PAC) fields see
for example \cite[Chapter 11]{FJ}.
We denote by
$\mathfrak{L}_{\mathrm{ring}}=\{+,\cdot,-,0,1\}$
the language of rings,
by
$\F$ be the $\mathfrak{L}_{\mathrm{ring}}$-theory of fields,
and by
$\PAC$ the $\mathfrak{L}_{\mathrm{ring}}$-theory of PAC fields.
We denote by $\rmPAC$ the context $(\mathrm{Sent}(\mathfrak{L}_{\mathrm{ring}}),\PAC)$
and accordingly by $\rmPAC_{\exists}$ and $\rmPAC_{\exists_1}$ the contexts
with fragments ${\rm Sent}_\exists(\mathfrak{L}_{\mathrm{ring}})$ respectively 
${\rm Sent}_{\exists_1}(\mathfrak{L}_{\mathrm{ring}})$.
Similarly, we define the contexts $\rmF$, $\rmF_\exists$ and $\rmF_{\exists_1}$.
Further we define the bridge
$\rmPAC_{\exists_1}\bridge\rmPAC = (\rmPAC_{\exists_1},\rmPAC,{\rm id})$,
and similarly the bridges
$\rmPAC_{\exists_1}\bridge\rmPAC_\exists$,
$\rmF_{\exists_1}\bridge\rmPAC$,
$\rmF_{\exists_1}\bridge\rmPAC_{\exists_1}$
and $\rmF_{\exists_1}\bridge\rmPAC_\exists$
(also with $\sigma=\mathrm{id}$).
Together with the interpretation $\iota={\rm id}$, we thus obtain the arches
$\rmF_{\exists_1}\bridge\rmPAC_{\exists_1}\arch\rmF_{\exists_1}\bridge\rmPAC$ and
$\rmPAC_{\exists_1}\bridge\rmPAC_\exists\arch\rmPAC\bridge\rmPAC$.

\begin{lemma}\label{prop:PAC}
Let $E,F$ be extensions of a field $K$ with $F$ PAC and $E/K$ separable.
We view $E$ and $F$ as $\mathfrak{L}_{\rm ring}(K)$-structures.
If ${\rm Th}_{\exists_1}(E)\subseteq{\rm Th}_{\exists_1}(F)$,
then ${\rm Th}_{\exists}(E)\subseteq{\rm Th}_{\exists}(F)$.
\end{lemma}

\begin{proof}
Let $E_0=E\cap\bar{K}$.
The assumption ${\rm Th}_{\exists_1}(E)\subseteq{\rm Th}_{\exists_1}(F)$ implies that there exists a $K$-embedding of $E_0$ into $F$ \cite[Lemma 20.6.3]{FJ},
so assume without loss of generality that $E_0\subseteq F$.
Since $E/K$ is separable, $E/E_0$ is regular.
This implies that $F\otimes_{E_0}E$ is an integral domain \cite[Ch.~V §17 Proposition 8]{Bourbaki_Algebra}, 
and so $L:={\rm Quot}(F\otimes_{E_0}E)$ is a field that is the compositum of the linearly disjoint extensions $F/E_0$ and $E/E_0$.
In particular, $L/F$ is regular,
so since $F$ is PAC, it is existentially closed in $L$ \cite[Proposition 11.3.5]{FJ}.
Thus, for the $\mathfrak{L}_{\rm ring}(K)$-theories we have ${\rm Th}_\exists(F)={\rm Th}_\exists(L)\supseteq{\rm Th}_\exists(E)$.
\end{proof}

\begin{lemma}\label{lem:PAC}
\begin{enumerate}[$(a)$]
    \item
	Hypotheses {\rm(}\ref{pattern.i}{\rm)}--{\rm(}\ref{pattern.iii}{\rm)}
	of Corollary~\ref{lem:pattern} hold for
	the arch $\rmPAC_{\exists_1}\bridge\rmPAC_\exists\arch\rmPAC\bridge\rmPAC$.
    \item
	Hypotheses {\rm(}\ref{pattern.i}{\rm)}--{\rm(}\ref{pattern.iii}{\rm)}
	of Corollary~\ref{lem:pattern} hold for 
	the arch $\rmF_{\exists_1}\bridge\rmPAC_{\exists_1}\arch\rmF_{\exists_1}\bridge\rmPAC$.
\end{enumerate}
\end{lemma}
\begin{proof}
For (\ref{pattern.i}): The fact that the fragments are computable is clear in both cases.
Just like the theory $\mathbf{F}$, the theory $\PAC$ is computably enumerable since it has a computable axiomatization,
see e.g.~\cite[Proposition 11.3.2]{FJ}.

For (\ref{pattern.ii}): In case (a) this is trivial.
For (b), for every $K\models\F$ there exists a regular extension $F/K$ which is PAC, by \cite[Proposition 13.4.6]{FJ}, so $F\models\PAC$ and ${\rm Th}_{\exists_1}({\rm id}(F))={\rm Th}_{\exists_1}(F)={\rm Th}_{\exists_1}(K)$.

For (\ref{pattern.iii}): 
In case (b) this is trivial.
For (a), given PAC fields $M,M'$ with ${\rm Th}_{\exists_1}(M)\subseteq{\rm Th}_{\exists_1}(M')$,
we have that ${\rm char}(M)={\rm char}(M')$,
so without loss of generality, $M$ and $M'$ have the same prime field $K$.
Since $K$ is perfect, Lemma \ref{prop:PAC} then implies that ${\rm Th}_{\exists}(M)\subseteq{\rm Th}_{\exists}(M')$.
\end{proof}

\begin{corollary}\label{cor:PAC}
\begin{enumerate}[$(a)$]
\item
\begin{enumerate}[{\rm(I)}]
\item
$\rmPAC_{\exists_{1}}\bridge\rmPAC_{\exists}$
admits a computable elimination
$\mathrm{Sent}_{\exists}(\mathfrak{L}_{\mathrm{ring}})\longrightarrow\mathrm{Sent}_{\exists_{1}}(\mathfrak{L}_{\mathrm{ring}})$.
\item
For every $R\subseteq\mathrm{Sent}(\mathfrak{L}_{\mathrm{ring}})$,
$(\PAC\cup R)_{\exists_{1}}\meq(\PAC\cup R)_{\exists}$.
\item
$\mathrm{Th}_{\exists}(K)=(\PAC\cup\mathrm{Th}_{\exists_1}(K))_{\exists}$, for each PAC field $K$.
\end{enumerate}
\item
\begin{enumerate}[{\rm(I)}]
\item[{\rm(II)}]
For every $R\subseteq\mathrm{Sent}_{\exists_{1}}(\mathfrak{L}_{\mathrm{ring}})$,
$(\F\cup R)_{\exists_{1}}\meq(\PAC\cup R)_{\exists_1}$.
\end{enumerate}
\end{enumerate}
\end{corollary}

\begin{remark}
Corollary \ref{cor:PAC}(a) is already implicit in \cite[Prop.~2.2]{Dries},
cf.~also \cite[Prop.~4.22]{DDF} for existential {\em formulas} over {\em perfect} PAC fields.
We could have defined the bridges $\rmF_{\exists_1}\bridge\rmPAC_\exists$
and $\rmF_{\exists_1}\bridge\rmPAC_{\exists_1}$ instead of with $\sigma={\rm id}$
also with $\sigma$ the map that sends a field $K$ to its algebraic part $K_{\rm alg}$,
the relative algebraic closure of the prime field of $K$ in $K$.
\end{remark}

Combining Corollary \ref{cor:PAC}(aII) and (bII), we get that
$(\F\cup R)_{\exists_1}\meq(\PAC\cup R)_{\exists}$ for every
$R\subseteq\mathrm{Sent}_{\exists_{1}}(\mathfrak{L}_{\mathrm{ring}})$.
We will generalize this in Corollary~\ref{cor:hash3},
using the following lemma that strengthens Corollary~\ref{cor:PAC}(bII).

\begin{lemma}\label{lem:PAC_F}
For $R\subseteq{\rm Sent}_\exists(\mathfrak{L}_{\rm ring})$,
$(\F\cup R)_{\exists_1}=(\PAC\cup R)_{\exists_1}$.
\end{lemma}

\begin{proof}
The inclusion $(\F\cup R)_{\exists_1}\subseteq(\PAC\cup R)_{\exists_1}$ is trivial.
For the other inclusion let $\varphi\in(\PAC\cup R)_{\exists_1}$ and $K\models\F\cup R$. By \cite[Proposition 13.4.6]{FJ} there exists a regular extension $F/K$ with $F$ PAC.
Since $R\subseteq{\rm Sent}_\exists(\mathfrak{L}_{\rm ring})$, we have $F\models\PAC\cup R$, hence $F\models\varphi$.
Moreover,
since $\varphi\in{\rm Sent}_{\exists_1}(\mathfrak{L}_{\rm ring})$
and $K$ is algebraically closed in $F$,
this implies $K\models\varphi$.
\end{proof}

\begin{remark}
We note that 
$(\F\cup R)_{\exists_1}=(\PAC\cup R)_{\exists_1}$ need not hold
for arbitrary $R\subseteq{\rm Sent}(\mathfrak{L}_{\rm ring})$,
as the case $R={\rm Th}(K)$ for any non-PAC field $K$ shows,
and one can also see that in general not even
$(\F\cup R)_{\exists_1}\Teq(\PAC\cup R)_{\exists_1}$ must hold.
\end{remark}

Combining
Corollary \ref{cor:PAC}(aII) with
Lemma \ref{lem:PAC_F},
we obtain the following first main results on PAC fields. More conclusions of Lemma \ref{lem:PAC} will be drawn in Section \ref{sec:PAC2}.

\begin{corollary}\label{cor:hash3}
For $R\subseteq{\rm Sent}_\exists(\mathfrak{L}_{\rm ring})$,
$(\F\cup R)_{\exists_1}\meq(\PAC\cup R)_\exists$.
\end{corollary}

\begin{remark}
We briefly describe another approach to obtain
Corollary~\ref{cor:hash3}.
First
we consider a weakening of
the hypothesis (\ref{pattern.ii})
of Corollary~\ref{lem:pattern}
to the following:
\begin{enumerate}
\item[(\ref{pattern.ii}')]
for every $N\models T_{1}$,
there exists $M\models T_{2}$
such that $\mathrm{Th}_{\hat{L}_{1}}(N)\subseteq\mathrm{Th}_{\hat{L}_{1}}(\sigma M)$
and
$\mathrm{Th}_{L_{1}}(N)=\mathrm{Th}_{L_{1}}(\sigma M)$.
\end{enumerate}
Next one can prove an adaptation of Corollary~\ref{lem:pattern}:
Suppose an arch $(B,\hat{B},\iota)$ satisfies
the hypotheses
(\ref{pattern.i},\ref{pattern.ii}',\ref{pattern.iii})
of Corollary~\ref{lem:pattern}.
Then $B$ admits a computable elimination
and
$(T_{1}\cup R)_{L_{1}}\meq(T_{2}\cup\iota R)_{L_{2}}$,
for each $R\subseteq\hat{L}_{1}$.
One can show that 
$\check{B}$ (defined as in Lemma \ref{lem:sur_goes_up}) satisfies \axsurj\
using only (\ref{pattern.ii}') and not the full strength of (\ref{pattern.ii}).
Finally, using again  \cite[Proposition 13.4.6]{FJ} 
one sees that
hypotheses (\ref{pattern.i},\ref{pattern.ii}',\ref{pattern.iii})
of Corollary~\ref{lem:pattern}
hold for 
$\rmF_{\exists_{1}}\bridge\rmPAC_{\exists}\arch\rmF_{\exists}\bridge\rmPAC$.
\end{remark}

\begin{remark}
While it is known by work of Cherlin, van den Dries and Macintyre
that the theory of a PAC field is determined by its algebraic part, its degree of imperfection and the theory of the inverse system of its absolute Galois group,
we see here that for the existential theory only the first of these three plays a role.
If we would restrict our attention to PAC fields that are for example perfect and $\omega$-free (equivalently, Hilbertian), we could obtain computable eliminations even for the full theory, 
but since this is essentially well known and our focus in this paper are existential theories, we do not pursue this
and only remark that Corollary \ref{cor:PAC}(aII) readily applies
to the theory $\mathbf{hPAC}$ of Hilbertian PAC fields,
and also Lemma \ref{lem:PAC_F} holds with $\PAC$ replaced by $\mathbf{hPAC}$
since the quoted \cite[Proposition 13.4.6]{FJ}
actually gives a Hilbertian PAC field $F$.
\end{remark}

\subsection{Equicharacteristic henselian valued fields: $\exists$-theories}
\label{sec:HeE}
For basics on valued fields see for example \cite{EP}.
We denote by
\begin{align*}
    \mathfrak{L}_{\mathrm{val}}
    &=
    \{+^{\mathbf{K}},-^{\mathbf{K}},\cdot^{\mathbf{K}},0^{\mathbf{K}},1^{\mathbf{K}},+^{\mathbf{k}},-^{\mathbf{k}},\cdot^{\mathbf{k}},0^{\mathbf{k}},1^{\mathbf{k}},+^{\bfGamma},<^{\bfGamma},0^{\bfGamma},\infty^{\bfGamma},v,\res\}
\end{align*}
the (three-sorted) language of valued fields,
with a sort $\mathbf{K}$ for the field itself, a sort $\bfGamma$ for the value group together with infinity, and a sort $\mathbf{k}$ for the residue field,
as well as symbols $v$ for the valuation map and $\res$ for the residue map.
We denote by $\mathfrak{L}_{\mathrm{val}}(\varpi)$ the expansion of $\mathfrak{L}_{\mathrm{val}}$ by a new constant symbol $\varpi$ in the sort $\mathbf{K}$.
A valued field $(K,v)$ gives rise in the usual way to an $\mathfrak{L}_{\rm val}$-structure 
$$
 (K,vK\cup\{\infty\},Kv,v,\mathrm{res}),
$$
where $vK$ is the value group, $Kv$ is the residue field, and $\mathrm{res}$ is the residue map
(set to $0\in Kv$ outside the valuation ring).
For notational simplicity, we will usually write $(K,v)$ to refer to the $\mathfrak{L}_{\rm val}$-structure it induces.

\begin{remark}\label{rem:one-sorted}
We might instead work in a one-sorted language of valued fields
$\mathfrak{L}_{O}=\mathfrak{L}_{\rm ring}\cup\{O\}$,
where $O$ is a unary predicate
and a valued field $(K,v)$ gives rise to an $\mathfrak{L}_{O}$-structure
by interpreting $O$ by the valuation ring $\mathcal{O}_{v}$.
As  
there is a uniform biinterpretation $\Gamma$ between the
$\mathfrak{L}_{\rm val}$-structure $(K,v)$
and the $\mathfrak{L}_O$-structure $(K,\mathcal{O}_v)$
in which each defining formula $\varphi_\Gamma$ (see the introduction) is equivalent modulo the theory of valued fields to both an existential and a universal formula,
all our results for the fragment ${\rm Sent}_{\exists}(\mathfrak{L}_{\rm val})$  
carry over to ${\rm Sent}_{\exists}(\mathfrak{L}_O)$.
\end{remark}

We denote by
$\bHen$
the $\mathfrak{L}_{\mathrm{val}}$-theory of equicharacterstic henselian nontrivially valued fields,
and by
$\bHepi$
(resp.,
$\bHepiz$)
the $\mathfrak{L}_{\mathrm{val}}(\varpi)$-theory of equicharacteristic
(resp., equicharacteristic zero)
henselian valued fields in which the interpretation of $\varpi$ is a uniformizer
(i.e.~an element of minimal positive value).
We define the contexts
$\Hen:=(\mathrm{Sent}(\mathfrak{L}_{\mathrm{val}}),\bHen)$
and 
$\Hen_{\exists}:=(\mathrm{Sent}_{\exists}(\mathfrak{L}_{\mathrm{val}}),\bHen)$,
and analogously 
$\Hepi$ and $\Hepiz$ with fragment ${\rm Sent}(\mathfrak{L}_{\rm val}(\varpi))$, and
$\Hepi_{\exists}$ and $\HepizE$
with fragment ${\rm Sent}_\exists(\mathfrak{L}_{\rm val}(\varpi))$.

We denote by $\F_{0}$ the theory of fields of characteristic zero and define the contexts
$\rmF_{0}=(\mathrm{Sent}(\mathfrak{L}_{\mathrm{ring}}),\F_{0})$
and analogously 
$\rmF_{0,\exists}$.
We denote by $\sigma_{\mathbf{k}}$
the map that sends a valued field $(K,v)$
to its residue field $Kv$.
Thus by equipping each with an appropriate restriction of $\sigma_{\mathbf{k}}$ we have the bridges
$\rmF\bridge\Hen$,
$\rmF_{\exists}\bridge\Hen_{\exists}$,
$\rmF_{0}\bridge\Hepiz$,
$\rmF_{0,\exists}\bridge\HepizE$,
$\rmF\bridge\Hepi$, and
$\rmF_{\exists}\bridge\Hepi_{\exists}$.

\begin{definition}\label{def:residue_interpretation}
For an 
$\mathfrak{L}_{\mathrm{ring}}$-term $t$
we define an 
$\mathfrak{L}_{\mathrm{val}}$-term $t_{\mathbf{k}}$
by recursion on the length of $t$ as follows:
for a variable $x$ we let $x_{\mathbf{k}}$
be a variable of the residue field sort;
we let $0_{\mathbf{k}}:=0^\mathbf{k}$ and $1_{\mathbf{k}}:=1^{\mathbf{k}}$;
for $t=t_1+t_2$ we let $t_{\mathbf{k}}=(t_{1})_{\mathbf{k}}+^{\mathbf{k}}(t_{2})_{\mathbf{k}}$,
similarly for $t=t_1-t_2$ and $t=t_1\cdot t_2$.
For an $\mathfrak{L}_{\mathrm{ring}}$-formula $\varphi$
we define an 
$\mathfrak{L}_{\mathrm{val}}$-formula $\varphi_{\mathbf{k}}$
by recursion on the length of $\varphi$ as follows:
if $\varphi= t_1\stackrel{.}{=}t_2$, we let
$\varphi_{\mathbf{k}}=(t_1)_\mathbf{k}\stackrel{.}{=}(t_2)_\mathbf{k}$;
if $\varphi=\neg\psi$, we define
$\varphi_{\mathbf{k}}:=\neg\psi_{\mathbf{k}}$;
if $\varphi=(\psi\wedge\rho)$, we define
$\varphi_{\mathbf{k}}:=(\psi_{\mathbf{k}}\wedge\rho_{\mathbf{k}})$;
and if $\varphi=\exists x\;\psi$, we define
$\varphi_{\mathbf{k}}:=\exists x\in\mathbf{k}\;\psi_{\mathbf{k}}$.
Finally, define
$\iota_{\mathbf{k}}:\mathrm{Sent}(\mathfrak{L}_{\mathrm{ring}})\longrightarrow\mathrm{Sent}(\mathfrak{L}_{\mathrm{val}})\subseteq\mathrm{Sent}(\mathfrak{L}_{\mathrm{val}}(\varpi))$,
given by
$\varphi\longmapsto\varphi_{\mathbf{k}}$.
\end{definition}

\begin{remark}
The map
$\iota_{\mathbf{k}}$ is an interpretation for the bridge
$(\rmF,\rmVF,\sigma_{\mathbf{k}})$,
where
$\VF$
is the
$\mathfrak{L}_{\mathrm{val}}$-theory of valued fields,
and
$\rmVF$ is the context
$(\mathrm{Sent}(\mathfrak{L}_{\mathrm{val}}),\VF)$. 
Equipped with the appropriate restrictions of $\iota_{\mathbf{k}}$ we get arches
$\rmF_{\exists}\bridge\Hen_{\exists}\arch\rmF\bridge\Hen$, 
$\rmF_{0,\exists}\bridge\HepizE\arch\rmF_{0}\bridge\Hepiz$
and
$\rmF_{\exists}\bridge\Hepi_{\exists}\arch\rmF\bridge\Hepi$,
where we use that
$\iota_{\mathbf{k}}(\mathrm{Sent}_{\exists}(\mathfrak{L}_{\mathrm{ring}}))\subseteq\mathrm{Sent}_{\exists}(\mathfrak{L}_{\mathrm{val}})\subseteq\mathrm{Sent}_{\exists}(\mathfrak{L}_{\mathrm{val}}(\varpi))$.
\end{remark}

\begin{remark}\label{rem:R4}
A field $K$ is large if it is existentially closed in $K(\!(t)\!)$,
see \cite{Pop,BSF,Popsurvey} for background on large fields.
To apply our machinery to the arch
$\rmF_{\exists}\bridge\Hepi_{\exists}\arch\rmF\bridge\Hepi$,
more precisely to verify \axmono\ for $\rmF_{\exists}\bridge\Hepi_{\exists}$,
we will work under the following hypothesis, introduced in \cite[\S2]{ADF22}:
\begin{enumerate}
	\item[{\bf(R4)}]\label{R4}\label{R4_F}
	Every large field $K$ is existentially closed in every extension $F/K$ for which there exists a valuation $v$ on $F/K$ with residue field $Fv=K$.
\end{enumerate}
As discussed in \cite[Proposition 2.3]{ADF22},
\Rfour\
is a consequence of Local Uniformization, which is in turn implied by Resolution of Singularities.
For a ring $C$ we also consider the hypothesis
{\bf(R4)$_{C}$}, which we define as \Rfour\ restricted to large fields $K$ that admit a ring homomorphism $C\rightarrow K$.
Since \Rfour\ holds when restricted to perfect fields $K$ \cite[Remark 2.4]{ADF22}, in particular
\RfourF{\mathbb{Q}} holds.
\end{remark}

\begin{lemma}\label{lem:Hen_sat_axioms}
\quad
\begin{enumerate}[$(a)$]
	\item
         Hypotheses {\rm(}\ref{pattern.i}{\rm)}--{\rm(}\ref{pattern.iii}{\rm)}
         of
	Corollary~\ref{lem:pattern}
	hold for
	$\rmF_{\exists}\bridge\Hen_{\exists}\arch\rmF\bridge\Hen$.
	\item
	Hypotheses {\rm(}\ref{pattern.i}{\rm)}--{\rm(}\ref{pattern.iii}{\rm)}
        of Corollary~\ref{lem:pattern}
	hold for 
	$\rmF_{0,\exists}\bridge\HepizE\arch\rmF_{0}\bridge\Hepiz$.
	\item
        Hypotheses {\rm(}\ref{pattern.i}{\rm)}--{\rm(}\ref{pattern.ii}{\rm)}
        of Corollary~\ref{lem:pattern}
	hold for 
	$\rmF_{\exists}\bridge\Hepi_{\exists}\arch\rmF\bridge\Hepi$,
	and if in addition \Rfour\ holds,
	then so does
        hypothesis
        {\rm(}\ref{pattern.iii}{\rm)}.
\end{enumerate}
\end{lemma}
\begin{proof}
For (\ref{pattern.i}):
The fact that the fragments 
and $\iota_\mathbf{k}$
are computable is clear in both cases. The theories
$\bHen$ and $\bHepi$ are computably enumerable since they have computable axiomatizations,
see e.g.~\cite[Remark 4.8]{ADF22},
taking Remark~\ref{rem:one-sorted} into account.

For (\ref{pattern.ii}):
let $k$ be a field.
Then $(k(\!(t)\!),v_{t})$ is an equicharacteristic henselian nontrivially valued field with residue field $k$, of which $t$ is a uniformizer.
In particular,
$\sigma_\mathbf{k}:\mathrm{Mod}(\bHepi)\longrightarrow\mathrm{Mod}(\F)$ is surjective,
and so (\ref{pattern.ii}) follows in both cases.

For (\ref{pattern.iii}) in (a):
let $(K,v),(L,w)\models\bHen$
and
suppose that
$Kv\models\mathrm{Th}_{\exists}(Lw)$.
That is
$(K,v)\models\bHen\cup\iota_{\mathbf{k}}\mathrm{Th}_{\exists}(Lw)$.
By \cite[Lemma 6.3]{AF16}%
\footnote{Alternatively, one may apply
Proposition~\ref{prp:intersection} taking
$R=\mathrm{Th}(Lw)$.},
$\bHen\cup\iota_{\mathbf{
k}}\mathrm{Th}_{\exists}(Lw)\models(\bHen\cup\iota_{\mathbf{
k}}\mathrm{Th}(Lw))_{\exists}$.
By \cite[Lemma 6.1]{AF16},
$\bHen\cup\iota_{\mathbf{
k}}\mathrm{Th}(Lw)\models\mathrm{Th}_{\exists}(L,w)$.
Then
$(K,v)\models\mathrm{Th}_{\exists}(L,w)$,
and therefore
$\rmF_{\exists}\bridge\Hen_{\exists}$ satisfies \axmono.

For (\ref{pattern.iii}) in (c):
\Rfour\
implies that $\rmF_{\exists}\bridge\Hepi_{\exists}$ satisfies \axmono\ by \cite[Proposition 4.11]{ADF22} (again, note  Remark~\ref{rem:one-sorted}),
taking $C$ to be the subfield $\mathbb{F}(\pi)$ generated by the uniformizer $\pi$ and the prime subfield $\mathbb{F}$, together with the $\pi$-adic valuation $v_{\pi}$.
Note that the valuation ring of $v_{\pi}$ is excellent.

For (\ref{pattern.iii}) in (b):
as for (c) except applying 
\RfourF{\mathbb{Q}}
instead of \Rfour:
by \cite[Remark 4.18]{ADF22}
we may apply \cite[Proposition 4.11]{ADF22} to $(C,u)=(\mathbb{Q}(\pi),v_{\pi})$,
noting that
the valuation ring of $v_{\pi}$ is excellent,
the separability hypotheses are satisfied in characteristic zero.
\end{proof}

\begin{corollary}\label{cor:henselian}
\phantom{\quad}
\begin{enumerate}[$(a)$]
\item
    \begin{enumerate}[{\rm(I)}]
\item
	$\rmF_{\exists}\bridge\Hen_{\exists}$
	admits a computable elimination
	$\mathrm{Sent}_{\exists}(\mathfrak{L}_{\mathrm{val}})\longrightarrow\mathrm{Sent}_{\exists}(\mathfrak{L}_{\mathrm{ring}})$.
\item
	For every $R\subseteq{\rm Sent}(\mathfrak{L}_{\mathrm{ring}})$,
	$(\F\cup R)_{\exists}\meq(\bHen\cup\iota_{\mathbf{k}}R)_{\exists}$.
\item
	For every $(K,v)\models\bHen$, $\mathrm{Th}_{\exists}(K,v)=(\bHen\cup\iota_{\mathbf{k}}\mathrm{Th}_\exists(Kv))_{\exists}$.
\end{enumerate}
\item
\begin{enumerate}[{\rm(I)}]
\item
	$\rmF_{0,\exists}\bridge\HepizE$
	admits a computable elimination
	$\mathrm{Sent}_{\exists}(\mathfrak{L}_{\mathrm{val}}(\varpi))\longrightarrow\mathrm{Sent}_{\exists}(\mathfrak{L}_{\mathrm{ring}})$.
\item
	For every $R\subseteq{\rm Sent}(\mathfrak{L}_{\mathrm{ring}})$,
	$(\F_{0}\cup R)_{\exists}\meq(\bHepiz\cup\iota_{\mathbf{k}}R)_{\exists}$.
\item
	For every  $(K,v,\pi_{K})\models\bHepiz$, $\mathrm{Th}_{\exists}(K,v,\pi_{K})=(\bHepiz\cup\iota_{\mathbf{k}}\mathrm{Th}_\exists(Kv))_{\exists}$.
\end{enumerate}
\item
	Suppose \Rfour.
\begin{enumerate}[{\rm(I)}]
\item
	$\rmF_{\exists}\bridge\Hepi_{\exists}$
	admits a computable elimination
	$\mathrm{Sent}_{\exists}(\mathfrak{L}_{\mathrm{val}}(\varpi))\longrightarrow\mathrm{Sent}_{\exists}(\mathfrak{L}_{\mathrm{ring}})$.
\item
	For every $R\subseteq{\rm Sent}(\mathfrak{L}_{\mathrm{ring}})$,
	$(\F\cup R)_{\exists}\meq(\bHepi\cup\iota_{\mathbf{k}}R)_{\exists}$.
\item
	For every $(K,v,\pi_{K})\models\bHepi$, $\mathrm{Th}_{\exists}(K,v,\pi_{K})=(\bHepi\cup\iota_{\mathbf{k}}\mathrm{Th}_\exists(Kv))_{\exists}$.
\end{enumerate}
\end{enumerate}
\end{corollary}

\begin{remark}
By Lemma \ref{lem:neg},
Lemma \ref{lem:Hen_sat_axioms}
and therefore also each of the parts of
Corollary \ref{cor:henselian}
remains true if we replace each $\exists$
by $\forall/\exists$,
denoting,
in each of the respective languages,
the set of finite conjunctions and disjunctions
of existential or universal sentences.
For example,
for every $R\subseteq{\rm Sent}(\mathfrak{L}_{\mathrm{ring}})$,
$(\F\cup R)_{\forall/\exists}\meq(\bHen\cup\iota_{\mathbf{k}}R)_{\forall/\exists}$.
Similar extensions hold for several of the results in the rest of this paper,
but we will not spell them out each time.
\end{remark}

Although we have only proved \axmono\ for the bridge
$\rmF_{\exists}\bridge\Hepi_{\exists}$
under the hypothesis \Rfour,
we are able to prove \axwmono\ unconditionally:

\begin{lemma}\label{lem:Hen_sat_weak_axioms}
The fragment extension 
$\rmF_{\exists}\bridge\Hepi_{\exists}
\sqsubseteq
\rmF\bridge\Hepi$
satisfies \axwmono.
\end{lemma}

\begin{proof}
Let $(K,v,\pi)\models\bHepi$ and $k\models\F$
with $\mathrm{Th}_{\exists}(k)\subseteq\mathrm{Th}_{\exists}(Kv)$.
Without loss of generality, $(K,v,\pi)$ is $|k|^+$-saturated,
and then so is $Kv$.
This implies that there exists an embedding $\varphi:k\rightarrow Kv$,
see \cite[Lemma 5.2.1]{CK}.
By \cite[Lemma 4.4]{ADF22}
there exists an elementary extension $(K,v)\prec (K^*,v^*)$
with a partial section $\zeta:Kv\rightarrow K^*$ of ${\rm res}_{v^*}$.
The composition $\zeta\circ\varphi:k\longrightarrow K^{*}$
is an embedding of fields,
and moreover is an embedding of valued fields when $k$ is equipped with the trivial valuation.
Mapping $t\mapsto\pi$ and using the universal property of henselisations,
we may extend $\zeta\circ\varphi$ to an embedding of valued fields
$(k(t)^{h},v_{t},t)\longrightarrow(K^{*},v^{*},\pi)$.
Thus
$\mathrm{Th}_{\exists}(k(t)^{h},v_{t},t)\subseteq\mathrm{Th}_{\exists}(K^{*},v^{*},\pi_{K})=\mathrm{Th}_{\exists}(K,v,\pi_{K})$.
Note that the residue field of $(k(t)^{h},v_{t},t)$ is $k$,
in particular it is elementarily equivalent to $k$.
\end{proof}

\subsection{Equicharacteristic henselian valued fields: $\exists_{n}$-theories}
\label{sec:HeEn}

For $n\geq0$, we
define the contexts
$\Hen_{\exists_{n}}:=
(\mathrm{Sent}_{\exists_{n}}(\mathfrak{L}_{\mathrm{val}}),\bHen)$
and
$\rmF_{\exists_{n}}:=
(\mathrm{Sent}_{\exists_{n}}(\mathfrak{L}_{\mathrm{ring}}),\F)$,
as well as the bridge
$\rmF_{\exists_{n}}\bridge\Hen_{\exists_{n}}$.
As
$\iota_{\mathbf{k}}(\mathrm{Sent}_{\exists_{n}}(\mathfrak{L}_{\mathrm{ring}}))\subseteq\mathrm{Sent}_{\exists_{n}}(\mathfrak{L}_{\mathrm{val}})$,
we obtain an
arch
$\rmF_{\exists_{n}}\bridge\Hen_{\exists_{n}}\arch\rmF\bridge\Hen$.

\begin{lemma}\label{lem:E}
Let $\mathfrak{L}$ be a language and let $M,N$ be $\mathfrak{L}$-structures.
If $\mathfrak{L}$ contain no constant symbols assume that $n\geq 1$.
\begin{enumerate}[$(a)$]
\item
$\mathrm{Th}_{\exists}(M)\subseteq\mathrm{Th}_{\exists}(N)$ if and only if $\mathrm{Th}_{\exists}(M')\subseteq\mathrm{Th}_{\exists}(N)$ for every finitely generated substructure $M'\subseteq M$.
\item
$\mathrm{Th}_{\exists_{n}}(M)\subseteq\mathrm{Th}_{\exists_{n}}(N)$ if and only if $\mathrm{Th}_{\exists}(M')\subseteq\mathrm{Th}_{\exists}(N)$ for every substructure $M'
\subseteq M$ generated by at most $n$ elements.
\end{enumerate}
\end{lemma}

\begin{proof}
Part (a) follows immediately from (b), and the latter is what we prove:
For $\Leftarrow$ let $\varphi\in{\rm Th}_{\exists_n}(M)$.
Without loss of generality, $\varphi$ is of the form
$\exists y_1,\dots,y_n \psi(\underline{y})$ with $\psi$ quantifier-free.
Then $M\models\psi(\underline{b})$ for some $\underline{b}\in M^n$,
and $\underline{b}$ generates a substructure $M'$ of $M$ with $M'\models\varphi$, hence $\varphi\in{\rm Th}_{\exists_n}(M')\subseteq{\rm Th}_{\exists_n}(N)$.

For $\Rightarrow$, let $M'\subseteq M$ be generated by $a_1,\dots,a_n$, and let $\varphi\in{\rm Th}_{\exists}(M')$, without loss of generality
of the form $\exists y_1,\dots,y_m\psi(\underline{y})$ for some $m$.
Then $M'\models\psi(\underline{b})$ for some $b_1,\dots,b_m\in M'$,
and $b_i=t_i(a_1,\dots,a_n)$ for an $\mathfrak{L}$-term $t_i$, for each $i$.
Thus $\exists x_1,\dots,x_n\psi(t_1(\underline{x}),\dots,t_m(\underline{x}))\in{\rm Th}_{\exists_n}(M)\subseteq{\rm Th}_{\exists_n}(N)$,
and so there are $c_1,\dots,c_n\in N$ with
$N\models\psi(t_1(\underline{c}),\dots,t_m(\underline{c}))$,
in particular $\varphi\in{\rm Th}_\exists(N)$.
\end{proof}

\begin{lemma}\label{lem:Hen_n_sat_axioms}
Hypotheses {\rm(}\ref{pattern.i}{\rm)}--{\rm(}\ref{pattern.iii}{\rm)}
of Corollary~\ref{lem:pattern}
hold for
$\rmF_{\exists_{n}}\bridge\Hen_{\exists_{n}}\arch\rmF\bridge\Hen$.
\end{lemma}
\begin{proof}
For (\ref{pattern.i}):
this is again clear.

For (\ref{pattern.ii}):
this is again clear, see
Lemma~\ref{lem:Hen_sat_axioms}(a).

For (\ref{pattern.iii}):
Let $(K,v),(L,w)\models\bHen$
and suppose
$\mathrm{Th}_{\exists_{n}}(Kv)\subseteq\mathrm{Th}_{\exists_{n}}(Lw)$.
Without loss of generality, $(L,w)$ is $\aleph_0$-saturated.
We will show that ${\rm Th}_\exists(M')\subseteq{\rm Th}_\exists(L,w)$
for every $\mathfrak{L}_{\rm val}$-substructure $M'$ of $(K,v)$ generated by at most $n$ elements,
which
by Lemma~\ref{lem:E}(b)
will imply that
$\mathrm{Th}_{\exists_{n}}(K,v)\subseteq\mathrm{Th}_{\exists_{n}}(L,w)$.
Without loss of generality assume that
$M'$ is generated by
at most $n$ elements
in the sort $\mathbf{K}$,
in particular $R_0=M'^{\mathbf{K}}$ is a ring generated by
at most $n$ elements.
Let $E={\rm Quot}(R_0)$, let $u=v|_E$
and note that ${\rm Th}_\exists(M')\subseteq{\rm Th}_{\exists}(E,u)$.

First we consider the case that $u$ is trivial.
Then $(E,u)$ is isomorphic to the trivially valued field $(Eu,v_{\mathrm{triv}})$,
and this isomorphism carries $R_{0}$ to a subring of $Eu$, which we denote $R_{0}u$.
Indeed, $R_{0}u$ is a subring of $Kv$ generated by
at most $n$ elements,
so $\mathrm{Th}_{\exists}(R_{0}u)\subseteq\mathrm{Th}_{\exists}(Lw)$
by Lemma~\ref{lem:E}(b).
As $Lw$ is $\aleph_0$-saturated and $R_0$ is countable,
there exists an $\mathfrak{L}_{\rm ring}$-embedding $R_{0}u\longrightarrow Lw$
(\cite[Lemma 5.2.1]{CK}).
 This extends to an $\mathfrak{L}_{\mathrm{ring}}$-embedding $\eta:Eu=\mathrm{Quot}(R_{0}u)\longrightarrow Lw$,
	which already shows that $\mathrm{Th}_{\exists}(Eu)\subseteq\mathrm{Th}_{\exists}(Lw)$.
	Since $\eta(Eu)$ is a finitely generated field extension of its prime field,
	it is separably generated,
	so there is an embedding
	$\zeta:\eta(Eu)\longrightarrow L$
	such that $\res_{w}\circ\zeta=\mathrm{id}_{\eta(Eu)}$,
	see \cite[Lemma 2.3]{AF16}.
It follows that $w$ is trivial on the image of $\eta(Eu)$,
thus we have an $\mathfrak{L}_{\mathrm{val}}$-embedding
$(\eta(Eu),v_{\mathrm{triv}})\longrightarrow(L,w)$.
This shows that
${\rm Th}_\exists(E,u)=\mathrm{Th}_{\exists}(Eu,v_{\mathrm{triv}})=\mathrm{Th}_{\exists}(\eta(Eu),v_{\mathrm{triv}})\subseteq\mathrm{Th}_{\exists}(L,w)$.

Second we consider the case that $u$ is nontrivial.
Let $\mathbb{F}$ be the prime field of $Eu$,
and let $R\subseteq Eu$ be a finitely generated subring of $Eu$.
The field of fractions $F=\mathrm{Quot}(R)$ is a finitely generated subfield of $Eu$.
Since $u$ is nontrivial, $\mathrm{trdeg}(F/\mathbb{F})<\mathrm{trdeg}(E/\mathbb{F})\leq n$ by the Abhyankar inequality \cite[Theorem 3.4.3]{EP}.
Since $F/\mathbb{F}$ is separable and finitely generated,
there exists a separating transcendence base $T\subseteq F$.
Thus the field extension $F/\mathbb{F}(T)$ is finite and separable, and therefore generated by a single element $a$, by the primitive element theorem.
Let $S\subseteq F$ be the subring generated by $T\cup\{a\}$.
Then $S$ is a subring of $Kv$ 
that is generated by at most $|T|+1\leq n$ elements, and ${\rm Quot}(S)=F$.
Thus, by Lemma~\ref{lem:E}(b),
$\mathrm{Th}_{\exists}(S)\subseteq\mathrm{Th}_{\exists}(Lw)$,
 which as argued for $R_{0}u$ and $Eu$ in the previous paragraph, implies that
$\mathrm{Th}_{\exists}(F)\subseteq\mathrm{Th}_{\exists}(Lw)$.
We have thus shown that
${\rm Th}_{\exists}(R)\subseteq\mathrm{Th}_{\exists}(Lw)$
for every finitely generated $\mathfrak{L}_{\rm ring}$-substructure $R$ of $Eu$,
which by Lemma~\ref{lem:E}(a) gives that
$\mathrm{Th}_{\exists}(Eu)\subseteq\mathrm{Th}_{\exists}(Lw)$.
Since by Lemma~\ref{lem:Hen_sat_axioms}(a) $\rmF_{\exists}\bridge\Hen_{\exists}$ satisfies {\axmono},
we have
$\mathrm{Th}_{\exists}(E^{h},u^{h})\subseteq\mathrm{Th}_{\exists}(L,w)$,
where $(E^{h},u^{h})\supseteq(E,u)$ denotes the henselization,
which is also of equal characteristic  and has residue field $Eu$.
Therefore
$\mathrm{Th}_{\exists}(E,u)\subseteq\mathrm{Th}_{\exists}(L,w)$.
\end{proof}

\begin{corollary}\label{cor:En}
\begin{enumerate}[{\rm(I)}] 
\item
$\rmF_{\exists_{n}}\bridge\Hen_{\exists_{n}}$ admits a computable elimination
$\mathrm{Sent}_{\exists_{n}}(\mathfrak{L}_{\mathrm{val}})\longrightarrow\mathrm{Sent}_{\exists_{n}}(\mathfrak{L}_{\mathrm{ring}})$.
\item
For every $R\subseteq\mathrm{Sent}(\mathfrak{L}_{\mathrm{ring}})$,
$(\F\cup R)_{\exists_{n}}\meq(\bHen\cup\iota_{\mathbf{k}}R)_{\exists_{n}}$.
\item
	For every $(K,v)\models\bHen$, $\mathrm{Th}_{\exists_n}(K,v)=(\bHen\cup\iota_{\mathbf{k}}\mathrm{Th}_{\exists_n}(Kv))_{\exists_n}$.
\end{enumerate}
\end{corollary}

\begin{proposition}\label{prp:En_uniform}
$\rmF_{\exists}\bridge\Hen_{\exists}$ admits a computable elimination
$\epsilon:\mathrm{Sent}_{\exists}(\mathfrak{L}_{\mathrm{val}})\rightarrow\mathrm{Sent}_{\exists}(\mathfrak{L}_{\mathrm{ring}})$
that for each $n\in\mathbb{N}$ restricts to an elimination
$\mathrm{Sent}_{\exists_{n}}(\mathfrak{L}_{\mathrm{val}})\rightarrow\mathrm{Sent}_{\exists_{n}}(\mathfrak{L}_{\mathrm{ring}})$
of $\rmF_{\exists_{n}}\bridge\Hen_{\exists_{n}}$.
\end{proposition}
\begin{proof}
This is an adaptation of the proof of Proposition~\ref{prp:elimination}.
We use two uniform aspects of the computability of the fragments of existential sentences.
Firstly, the function
$e:\mathrm{Sent}_\exists(\mathfrak{L}_{\mathrm{val}})\rightarrow\mathbb{N}\cup\{0\}$,
$\psi\mapsto\min\{n\in\mathbb{N}\mid\psi\in\mathrm{Sent}_{\exists_{n}}(\mathfrak{L}_{\mathrm{val}})\}$
is computable.
Secondly, 
we may fix a computable function
$\mathbb{N}\times\mathbb{N}\rightarrow\mathrm{Sent}_{\exists}(\mathfrak{L}_{\mathrm{ring}})$,
$(n,m)\mapsto\varphi_{n,m}$,
such that $\varphi_{n,1},\varphi_{n,2},\dots$ is an enumeration of $\mathrm{Sent}_{\exists_{n}}(\mathfrak{L}_{\mathrm{ring}})$,
for each $n$.
We can also fix a computable enumeration
$P_1,P_2,\dots$
of the proofs from $\bHen$, since the latter is computably enumerable.
By Corollary~\ref{cor:En}(I),
given $\psi\in\mathrm{Sent}_{\exists}(\mathfrak{L}_{\mathrm{val}})$
there exist $l$ and $m$ such that
$P_{l}$ is a proof of $\bHen\vdash\psi\leftrightarrow\iota_{\mathbf{k}}\varphi_{e(\varphi),m}$.
We take such pair $(l,m)$
minimial with respect to the ordering (\ref{eqn:ordering})
and define $\epsilon\psi:=\varphi_{e(\psi),m}$.
\end{proof}

\subsection{Proof of Theorem~\ref{thm:introhens}}\label{proof:1.1}
Part (a) is
Corollary~\ref{cor:henselian}(aII),
part (b) is
Corollary~\ref{cor:henselian}(cII)
and 
part (c) is Corollary~\ref{cor:En}(II).

\section{Limit theories}
\label{section:algorithms}

\noindent
By analogy with the notion of characteristic in the theory of fields,
we study extensions of theories defined by a sequence of sentences,
including limit cases.
First, in Section~\ref{sec:4.1}, we introduce theories derived from the addition of these sentences,
among them
an ``all'' theory, an ``almost all'' theory, and a uniform theory.
Then, in Section~\ref{sec:4.2} we prove Turing reductions between various combinations of these theories, which will be used in key proofs in Section~\ref{sec:final}.
Finally, in Section~\ref{sec:4.3}, we apply to these theories the framework of bridges from Section~\ref{sec:bridges}
to obtain Proposition~\ref{prp:stratify}, which will be used in the proofs of Corollary~\ref{prp:PAC} and Proposition~\ref{cor:F_and_P_combined}.

\subsection{All, almost all, and uniform}
\label{sec:4.1}
Fix a language $\mathfrak{L}$
and
a sequence of
$\mathfrak{L}$-sentences
$(\rho_{n})_{n\in\mathbb{N}}$.

\begin{definition}\label{def:notations}
For any $\mathfrak{L}$-theory $T$,
we define 
\begin{eqnarray*}
T_{0}&:=&(T\cup\{\rho_{n}\mid n\in\mathbb{N}\})^\vdash\\
T_{n}&:=&(T\cup\{\rho_{1},\ldots,\rho_{n-1},\neg\rho_{n}\})^\vdash, \quad n\in\mathbb{N}\\
T_{>m}&:=&\bigcap_{n>m}T_{n},\quad m\in\mathbb{N}\cup\{0\}\\
T_{\gg0}&:=&\bigcup_{m\in\mathbb{N}}\bigcap_{n\geq m}T_{n}\\
    T_N &:=& \{(\varphi,n) \mid n\in N, \varphi\in T_n\},\quad N\subseteq\mathbb{N}\cup\{0\}.
\end{eqnarray*}
For an $\mathfrak{L}$-fragment $L$, by analogy with above we write $T_{N,L}=T_{N}\cap(L\times(\mathbb{N}\cup\{0\}))$.
\end{definition}

\begin{example}
We will apply this definition in Section \ref{sec:final} mainly in the case where $T$ is some theory of (possibly valued) fields and
$\rho_n$ is a quantifier-free sentence expressing that the characteristic is not $n$,
so that $T_0$, $T_n$, $T_{>m}$ and $T_{\gg0}$
are, respectively, the 
sets of sentences true in
all models of $T$ of characteristic zero, 
all models of $T$ of characteristic $n$,
all models of $T$ of characteristic greater than $m$, 
and all models of $T$ of sufficiently large positive characteristic.
\end{example}

Note that every model of $T$ is a model of $T_{n}$ for exactly one $n\in\mathbb{N}\cup\{0\}$.
The models of $T_{>0}$ and $T_{\gg0}$ can be described as follows:

\begin{proposition}\label{lem:pseudo}
Let $T$ be an $\mathfrak{L}$-theory,
and let $M$ be an $\mathfrak{L}$-structure.
\begin{enumerate}[(a)]
   \item $M\models T_{>0}$ if and only if $M\equiv \prod_{i\in I}M_i/\mathcal{U}$ for some ultrafilter $\mathcal{U}$ on a set $I$,
    and $M_i\models T_{n_i}$ for some $n_i\in\mathbb{N}$.
    \item $M\models T_{\gg0}$ if and only if $M\models T_{>0}\cup T_{0}$
 
\end{enumerate}
\end{proposition}

\begin{proof}
(a): $T_{>0}$ is the common theory of the models of $T_{n}$, for $n\in\mathbb{N}$,
so the implication $\Rightarrow$ follows from \cite[Exercise 4.1.18]{CK},
and $\Leftarrow$ follows from
Łoś's theorem
since
$T_{n_i}\supseteq T_{>0}$.

(b): $\Rightarrow$ is trivial from the definitions.
Suppose $M\models T_{>0}\cup T_{0}$.
Then $M\equiv\prod_{i\in I}M_{i}/\mathcal{U}$,
as in $(a)$.
For $l>0$,
$M\models\rho_{l}$,
therefore
by Łoś's theorem
$\{i\in I\mid n_{i}=l\}\notin\mathcal{U}$.
Thus for every $m\geq 0$,
$\{i\in I\mid n_{i}> m\}\in\mathcal{U}$,
and so $M\models T_{>m}$, again by Łoś's theorem.
\end{proof}

\begin{proposition}
    \label{lem:little_equivalences}
Let $T,S$ be $\mathfrak{L}$-theories
such that $\neg\rho\in L$ for every $\rho\in S$.
Then
    \begin{enumerate}[$(a)$]
    \item
    $(T_{L}\cup S)_{L}=(T\cup S)_{L}$.
    \end{enumerate}
Consequently, if $\neg\rho_n\in L$ for every $n\in\mathbb{N}$, then
    \begin{enumerate}[$(a)$]
    \setcounter{enumi}{1}
    \item
    $(T_{L})_{0,L}=T_{0,L}$,
    \item
    $(T_{L})_{n,L}=T_{n,L}$ for every $n\in\mathbb{N}$,
    \item
    $(T_{L})_{>m,L}=T_{>m,L}$ for every $m\in\mathbb{Z}_{\geq0}$,
    \item
    $(T_{L})_{\gg0,L}=T_{\gg0,L}$, and
    \item
    $(T_{L})_{N,L}= T_{N,L}$ for every $N\subseteq\mathbb{N}\cup\{0\}$.
    \end{enumerate}
\end{proposition}
\begin{proof}
For (a),
the inclusion $\subseteq$ is clear
since $(T\cup S)_{L}=(T^{\vdash}\cup S)_{L}$ and $T_{L}\subseteq T^{\vdash}$.
Now let $\varphi\in(T\cup S)_{L}$.
Then
there exist $\rho_{1},\ldots,\rho_{m}\in S$ such that
$T\cup\{\rho_{1},\ldots,\rho_{m}\}\vdash\varphi$,
and
equivalently
$T\vdash(\bigvee_{j=1}^{m}\neg\rho_{j}\big)\vee\varphi$.
Since $(\bigvee_{j=1}^{m}\neg\rho_{j}\big)\vee\varphi\in L$,
we have
$(\bigvee_{j=1}^{m}\neg\rho_{j}\big)\vee\varphi\in T_{L}$,
in particular
$T_{L}\cup\{\rho_{1},\ldots,\rho_{m}\}\vdash\varphi$.
Thus $\varphi\in(T_{L}\cup S)_{L}$,
which proves $(a)$.
Both (b) and (c) are special cases of (a),
and (d) and (e) follow easily.
Finally, (f) follows directly from (b) and (c).
\end{proof}

\subsection{Turing reductions}
We keep notations and definitions from the previous subsection.
\label{sec:4.2}

\begin{proposition}\label{prop:algorithms}
Let $T$ be an $\mathfrak{L}$-theory,
let $L$ be a computable $\mathfrak{L}$-fragment
with $\rho_n,\neg\rho_n\in L$ for every $n$,
and assume that 
the function $n\mapsto\rho_n$ is computable. Then
\begin{enumerate}[$(a)$]
    \item $T_{0,L}\Tred T_L\oplus T_{>0,L}\oplus T_{\gg0,L}$, and
    \item if $T_{>0,L}$ and $T_{\gg0,L}$ are decidable, then $T_L\meq T_{0,L}$.
\end{enumerate}
Let in addition 
$T'$ be an $\mathfrak{L}$-theory 
with $T'_L\subseteq T_L$
and
$T_{\gg0,L}=T'_{\gg0,L}$. Then
\begin{enumerate}[$(a)$]
 \setcounter{enumi}{2}
    \item $T_{>0,L}\Tred T_{\gg0,L}\oplus T'_{>0,L}\oplus T_{\mathbb{N},L}$, and
    \item  if $T'_{\gg0,L}$
    and $T'_{>0,L}$ are decidable,
then $T_{>0,L}\Teq T_{\mathbb{N},L}$ and
$T_L\Teq T_{0,L}\oplus T_{>0,L}\Teq T_{\mathbb{N}\cup\{0\},L}$.
\end{enumerate}
\end{proposition}

The rest of this subsection is devoted to the proof of Proposition \ref{prop:algorithms},
and for this we will denote\footnote{Note that e.g.~by $\Sigma_0$ we do not mean the first part of Definition \ref{def:notations}
applied to $T=\Sigma$,
instead $\Sigma_{0}$ is $(T_{L})_{0,L}$.}
$\Sigma:=T_L$, $\Sigma_0:=(T_L)_0\cap L$,
$\Sigma_n:=(T_L)_n\cap L$,
$\Sigma_{>m}:=(T_L)_{>m}\cap L$,
$\Sigma_{\gg0}:=(T_L)_{\gg0}\cap L$,
$\Sigma_N:=\{(\varphi,n):n\in N,\varphi\in\Sigma_n\}$.
For $\varphi\in L$ and $n\in\mathbb{N}\cup\{0\}$ we will (only in this section) denote
$$
 \varphi_n\;:=\;\varphi\vee\bigvee_{i=1}^n\neg\rho_i\;\in L
$$
where it is understood that $\varphi_{0}=\varphi$.

\begin{lemma}\label{lem:TT0}
We have
$\Sigma=\Sigma_0\cap\Sigma_{>0}$.
In particular, if $\Sigma_{>0}$ is decidable,
then $\Sigma\mred\Sigma_0$.
\end{lemma}

\begin{proof}
The inclusion from left to right is trivial. 
For the converse inclusion,
let $\varphi\in\Sigma_{0}\cap\Sigma_{>0}$.
From $\varphi\in\Sigma_0$ we see that
$\Sigma\vdash\varphi_n$ for some minimal $n$.
So if $n>0$, then
$\Sigma\vdash\varphi_{n-1}\vee\neg\rho_{n}$,
and
$\varphi\in\Sigma_{>0}\subseteq\Sigma_n$ implies that
$\Sigma\vdash\varphi_{n-1}\vee\rho_n$,
hence $\Sigma\vdash\varphi_{n-1}$,
contradicting the minimality of $n$.
Thus $n=0$ and $\varphi=\varphi_0\in\Sigma$.
\end{proof}

\begin{lemma}\label{thm:main}
We have
$\Sigma_{0}\Tred\Sigma\oplus\Sigma_{>0}\oplus\Sigma_{\gg 0}$.
Moreover, if $\Sigma_{>0}$ and $\Sigma_{\gg0}$ are decidable, then $\Sigma_0\mred\Sigma$.
\end{lemma}

\begin{proof}
Suppose that there are algorithms $A_{\Sigma},A_{\Sigma_{>0}},A_{\Sigma_{\gg0}}$ to decide the theories $\Sigma,\Sigma_{>0},\Sigma_{\gg0}$.
We outline an algorithm to decide $\Sigma_{0}$.

\begin{enumerate}[---]
\item
{\bf Input.}
$\varphi\in L$

\item
{\bf Step 1.}
Apply $A_{\Sigma_{\gg0}}$ to decide whether or not
	$\varphi\in\Sigma_{\gg0}$.
If `NO' then in particular we have $\varphi\notin\Sigma_{0}$,
so we {\bf output} `NO' and finish.
If else `YES' then we continue.

\item
{\bf Step 2.}
Since $\varphi\in\Sigma_{\gg0}$ there exists $m\in\mathbb{N}$ such that 
$\varphi\in\bigcap_{n\geq m}\Sigma_n$.
So since $\Sigma_n\models\neg\rho_n$
we obtain $\varphi_m\in\bigcap_{n>0}\Sigma_n=\Sigma_{>0}$.
Apply $A_{\Sigma_{>0}}$ to
$\varphi_0,\varphi_1,\dots$ until we find the minimal $m_0$ with
 $\varphi_{m_0}\in\Sigma_{>0}$.

\item
{\bf Step 3.}
Apply $A_{\Sigma}$ to decide whether or not
   $\varphi_{m_0}\in\Sigma$.
\begin{itemize}
\item If `YES' then in particular we have
$\varphi_{m_0}\in\Sigma_{0}$.
Since also
$\Sigma_{0}\vdash\bigwedge_{i=1}^{m_0}\rho_i$,
it follows that $\Sigma_{0}\vdash\varphi$,
i.e.~$\varphi\in\Sigma_0$.
We {\bf output} `YES' and finish.
\item If `NO', on the other hand, then $\varphi_{m_0}\notin\Sigma$.
Since we already know $\varphi_{m_0}\in\Sigma_{>0}$, 
it follows from Lemma~\ref{lem:TT0} that $\varphi_{m_0}\notin\Sigma_{0}$.
Therefore $\varphi\notin\Sigma_{0}$.
We {\bf output} `NO' and finish.
\qedhere
\end{itemize}
\end{enumerate}
This proves the Turing reduction.
If both $\Sigma_{>0}$ and $\Sigma_{\gg0}$ are decidable,
$A_\Sigma$ is the only oracle occurring in the above algorithm,
and it gives the many-one reduction
$\Sigma_{0}\mred\Sigma$.
\end{proof}

\begin{lemma}\label{lem:uniform_reduction}
$\Sigma_\mathbb{N}\mred\Sigma$ and $\Sigma_\mathbb{N}\mred\Sigma_{>0}$.
\end{lemma}

\begin{proof}
For $\varphi\in L$ and $n\in\mathbb{N}$, let $\psi=\varphi_{n-1}\vee\rho_n$.
Then $(\varphi,n)\in\Sigma_\mathbb{N}$ if and only if
$\psi\in\Sigma$.
Moreover, $\psi\in\Sigma$ if and only if $\psi\in\Sigma_{>0}$.
Indeed,  $\Sigma\subseteq\Sigma_{>0}\subseteq\Sigma_n$,
and $\Sigma\vdash\psi$ if and only if $\Sigma_n\vdash\psi$.
\end{proof}

\begin{lemma}\label{prop:positive_reduces_to_uniform}
Let $\Sigma'\subseteq\Sigma$
with $\Sigma'_{\gg0}=\Sigma_{\gg0}$.
Then $\Sigma_{>0}\Tred\Sigma_{\gg0}\oplus\Sigma'_{>0}\oplus\Sigma_\mathbb{N}$.
\end{lemma}

\begin{proof}
Suppose there are algorithms $A_{\Sigma_\mathbb{N}}$, $A_{\Sigma'_{>0}}$ and $A_{\Sigma_{\gg0}}$ to decide $\Sigma_\mathbb{N}$, $\Sigma'_{>0}$ and $\Sigma_{\gg0}$. We outline an algorithm to decide $\Sigma_{>0}$.
\begin{enumerate}[---]
\item
{\bf Input.}
$\varphi\in L$.
\item
{\bf Step 1.}
Apply $A_{\Sigma_{\gg0}}$ to decide whether or not
	$\varphi\in\Sigma_{\gg0}$.
If `NO' then in particular $\varphi\notin\Sigma_{>0}$,
so we {\bf output} `NO' and finish.
If else `YES' then we continue.

\item
{\bf Step 2.}
Since $\Sigma'_{\gg0}=\Sigma_{\gg0}\models\varphi$ there exists $m\in\mathbb{N}$ such that $\Sigma'_n\models\varphi$
for every $n\geq m$.
In particular, $\Sigma'_n\models\varphi_m$ for every $n> 0$,
i.e.~$\Sigma'_{>0}\models\varphi_m$.
Apply $A_{\Sigma'_{>0}}$ to
$\varphi_0,\varphi_1,\dots$ until we find the minimal $m_0$ with
$\varphi_{m_0}\in\Sigma'_{>0}$.
\item
{\bf Step 3.}
Since $\Sigma'\subseteq\Sigma$ and $\varphi_{m_0}\in\Sigma'_{>0}$, also 
$\varphi_{m_0}\in\Sigma_{>0}$,
hence $\varphi\in\bigcap_{n>m_0}\Sigma_n$.
We now use $A_{\Sigma_{\mathbb{N}}}$ 
to check if $(\varphi,1),\dots,(\varphi,m_0)\in\Sigma_\mathbb{N}$.
If yes then $\varphi\in\bigcap_{n>0}\Sigma_n=\Sigma_{>0}$ and we {\bf output} `YES', otherwise we {\bf output} `NO'.\qedhere
\end{enumerate}
\end{proof}

\begin{proof}[Proof of Proposition \ref{prop:algorithms}]
By Proposition~\ref{lem:little_equivalences}
it suffices to prove statements 
with $T_{0,L}$ replaced by $(T_L)_{0,L}=\Sigma_0$,
$T_{>0,L}$ replaced by
$(T_L)_{>0,L}=\Sigma_{>0}$,
etc.,
which we do now:

(a): This is Lemma \ref{thm:main}.

(b): 
If $\Sigma_{>0}$ and $\Sigma_{\gg0}$ are decidable,
$\Sigma\mred\Sigma_0$ follows from Lemma~\ref{lem:TT0},
and $\Sigma_0\mred\Sigma$
from Lemma \ref{thm:main}.

(c): This is Lemma \ref{prop:positive_reduces_to_uniform}.

(d): $\Sigma_{>0}\Tred\Sigma_{\mathbb{N}}$
is (c)
and
$\Sigma_{>0}\geq_{T}\Sigma_{\mathbb{N}}$ is Lemma~\ref{lem:uniform_reduction}.
The reduction
$\Sigma\Tred\Sigma_{0}\oplus\Sigma_{>0}$ is Lemma~\ref{lem:TT0},
and 
$\Sigma\geq_{\mathrm{T}}\Sigma_{>0}$ follows from Lemma~\ref{lem:uniform_reduction} and (c),
and
$\Sigma\oplus\Sigma_{>0}\geq_{\mathrm{T}}\Sigma_{0}$ from (a).
Thus $\Sigma\Teq\Sigma_{0}\oplus\Sigma_{>0}$.
Trivially
$\Sigma_{0}\oplus\Sigma_{>0}\Teq\Sigma_{\mathbb{N}\cup\{0\}}$.
\end{proof}

\subsection{Limit arches}
\label{section:stratify_arch}
\label{sec:4.3}

\noindent
Let $A=(B,\hat{B},\iota)$ be an arch.
We write
$B=(C_1,C_2,\sigma)$
and
$\hat{B}=(\hat{C}_{1},\hat{C}_{2},\sigma)$,
where $C_{i}=(L_{i},T_{i})$ is a subcontext of $\hat{C}_{i}=(\hat{L}_{i},T_{i})$, for $i=1,2$.
Let $(\rho_{n})_{n\in\mathbb{N}}$ be a sequence of $\mathfrak{L}_1$-sentences,
with $\rho_n,\neg\rho_{n}\in L_{1}$ for every $n$.
Let $(\iota\rho_{n})_{n\in\mathbb{N}}$ be the corresponding sequence of $\mathfrak{L}_{2}$-sentences.

For an $\mathfrak{L}_{1}$-theory $R$,
we write 
$T_{2}(R)=T_{2}\cup\iota(R_{\hat{L}_{1}})$.
We will use the notation introduced in Definition~\ref{def:notations},
first with $T=R$ and the sentences $(\rho_{n})_{n\in\mathbb{N}}$,
leading to theories
$R_{0}$, $R_{n}$, $R_{>m}$, $R_{\gg0}$,
and second with $T=T_{2}(R)$ and the sentences $(\iota\rho_{n})_{n\in\mathbb{N}}$,
leading to theories
$T_{2}(R)_{0}$, $T_{2}(R)_{n}$, $T_{2}(R)_{>m}$, $T_{2}(R)_{\gg0}$.
In particular we have theories
$T_{1,0}$, $T_{1,n}$, $T_{1,>m}$, $T_{1,\gg0}$,
$T_{2,0}$, $T_{2,n}$, $T_{2,>m}$, $T_{2,\gg0}$.
We moreover extend one of these notations to contexts and bridges:
For $i=1,2$,
let $C_{i,0}=(L_{i},T_{i,0})$
and note that $B_{0}:=(C_{1,0},C_{2,0},\sigma|_{\mathrm{Mod}(T_{2,0})})$ is a bridge.

\begin{remark}
Note that the deductive closure of $T_{2}(R)$ does not depend on the choice of interpretation $\iota$:
we have $M\models T_{2}(R)$ if and only if $M\models T_{2}$ and $\sigma M\models R_{\hat{L}_{1}}$.
Moreover,
if $R\subseteq\hat{L}_{1}$ then
$(T_{2}\cup\iota R)^{\vdash}
=T_{2}(R)^{\vdash}
=(T_{2}\cup\iota((T_{1}\cup R)_{\hat{L}_{1}}))^{\vdash}$
by Lemma~\ref{lem:sur_goes_up};
and
if instead $R\supseteq T_{1}$ then again
$T_{2}(R)^{\vdash}=(T_{2}\cup\iota((T_{1}\cup R)_{\hat{L}_{1}})^{\vdash}$.
\end{remark}

\begin{proposition}\label{prp:stratify}
\phantom{\quad}
	\begin{enumerate}[$(a)$]
	\item
	$(T_{2})^{\vdash}=T_{2}(T_{1})^{\vdash}$.
	\item
	$T_{2,n}=T_{2}(T_{1,n})^{\vdash}$, for all $n\in\mathbb{N}\cup\{0\}$.
	\item
	If $B$ satisfies \axsurj\ and $B_{0}$ satisfies \axmono,
	then
	$(T_{2,>m})_{L_{2}}=T_{2}(T_{1,>m})_{L_{2}}$, for all $m\geq0$.
	\item
	If $B$ satisfies \axsurj\ and $B_{0}$ satisfies \axmono,
	then
	$(T_{2,\gg0})_{L_{2}}=T_{2}(T_{1,\gg0})_{L_{2}}$.
	\end{enumerate}
\end{proposition}
\begin{proof}
\begin{enumerate}[$(a)$]
\item
We apply Lemma~\ref{lem:sur_goes_up} to the bridge $\hat{B}$ and $R=\emptyset$.
\item
	Let $n\in\mathbb{N}$.
	We have assumed that $\rho_{n},\neg\rho_{n}\in {L}_{1}\subseteq \hat{L}_1$,
	thus $\iota\rho_{n},\iota\neg\rho_{n}\in\hat{L}_{2}$,
	since $\iota$ is an interpretation for $\hat{B}$.
	For $M\models T_{2}$, we have 
$M\models\neg\iota\rho_{n}\Leftrightarrow \sigma M\models\neg\rho_{n}\Leftrightarrow M\models\iota\neg\rho_{n}$.
	Thus (*): $T_{2}\models(\neg\iota\rho_{n}\leftrightarrow\iota\neg\rho_{n})$.
    Together with $T_{2}(R)^{\vdash}
=(T_{2}\cup\iota((T_{1}\cup R)_{\hat{L}_{1}}))^{\vdash}$
for $R=\{\rho_1,\dots,\rho_{n-1},\neg\rho_n\}$
by Lemma~\ref{lem:sur_goes_up},
we obtain $T_{2,n}=T_{2}(T_{1,n})^{\vdash}$,
and similarly $T_{2,0}=T_{2}(T_{1,0})^{\vdash}$.
\item
	By (b),
	for each $n>m$ we have
     $T_{1,n}\supseteq T_{1,>m}$,
     and therefore
	$T_{2,n}=T_{2}(T_{1,n})^{\vdash}\supseteq T_{2}(T_{1,>m})^{\vdash}$.
	Thus
	$(T_{2,>m})_{L_2}\supseteq T_{2}(T_{1,>m})_{L_2}$.
 
	Conversely,
	let
	$M\models T_{2}(T_{1,>m})$
	so that
	$\sigma M\models(T_{1,>m})_{\hat{L}_{1}}$.
Note that $T_{1,>m}\models\rho_{1}\wedge\ldots\wedge\rho_{m}$
and $\rho_{1}\wedge\ldots\wedge\rho_{m}\in L_{1}\subseteq\hat{L}_{1}$.
Thus $\sigma M\models\rho_{1}\wedge\ldots\wedge\rho_{m}$,
and therefore also $M\models \iota\rho_1\wedge\ldots\wedge\iota\rho_m$
(Lemma \ref{lem:boolean}).
	There are two cases:
 
	First, if $M\models T_{2,n}$ for some $n>0$,
	then necessarily $n>m$, and so 
	$M\models T_{2,>m}$.
	
 Otherwise, in the second case, we have
	$M\models T_{2,0}$,
	whence $\sigma M\models(T_{1,0})_{\hat{L}_{1}}$ by (b).
	In particular $\sigma M\models\rho_{n}$, for all $n>0$.
	Since $\sigma M$ is  a model of
	$(T_{1,>m})_{\hat{L}_{1}}$,
	by Lemma~\ref{lem:extension}
	there exists $N\models T_{1,>m}$
	such that $\sigma M\models\mathrm{Th}_{\hat{L}_{1}}(N)$.
	Since $\neg\rho_{n}\in L_{1}\subseteq\hat{L}_{1}$
    and $\sigma M\not\models\neg\rho_{n}$,
    also  $N\not\models\neg\rho_{n}$, i.e.~$N\models\rho_{n}$, for all $n>0$.
 	It follows that $N\models T_{1,0}$.
	Let $S=T_{1}\cup\{\rho_{1},\ldots,\rho_{m}\}$
	and write
	$\sigma_{n}=\rho_{m+n}$, for $n\in\mathbb{N}$.
	We apply the notation
 from
 Definition~\ref{def:notations}
to the $\mathfrak{L}_{1}$-theory $S$
	and the sequence of $\mathfrak{L}_{1}$-sentences $(\sigma_{n})_{n\in\mathbb{N}}$.
	Then
	$S_{n}=T_{1,m+n}$
	for each $n>0$,
	and so
	$S_{>0}=T_{1,>m}$
	and
	$S_{0}=T_{1,0}$.
	It follows that
	$N\models S_{>0}\cup S_{0}$.
    By Proposition~\ref{lem:pseudo}(a),
	$N\equiv\prod_{i\in I}H_{i}/\mathcal{U}$
	for a set $I$, an
  ultrafilter $\mathcal{U}$ on $I$
	and $H_{i}\models S_{n_i}$
 for some $n_i\in\mathbb{N}$.
	By \axsurj\ for ${B}$,
	for each $i\in I$
	there exists
	$K_{i}\models T_{2}$
	with
	$\mathrm{Th}_{L_1}(H_{i})=\mathrm{Th}_{L_1}(\sigma K_{i})$.
	It follows that $\sigma K_{i}\models S_{n_{i}}$,
 hence $K_{i}\models T_{2,m+n_{i}}$, for each $i\in I$,
 and so $K_{i}\models T_{2,>m}$.
	Denote
	$K:=\prod_{i\in I}K_{i}/\mathcal{U}$.
        By Łoś's theorem,
        $\mathrm{Th}_{L_1}(\sigma K)=\mathrm{Th}_{L_1}(\prod_{i\in I}\sigma K_{i}/\mathcal{U})=\mathrm{Th}_{L_{1}}(N)$.
 	So since $N\models S_0$ it follows that
	$K\models T_{2,0}\cup T_{2,>m}$.
		Thus,
		$\sigma M\models\mathrm{Th}_{{L}_{1}}(\sigma K)$,
		in particular
		$\mathrm{Th}_{L_{1}}(\sigma K)\subseteq\mathrm{Th}_{L_{1}}(\sigma M)$.
		Since $B_{0}$ satisfies \axmono,
		we conclude that
		$\mathrm{Th}_{L_{2}}(K)\subseteq\mathrm{Th}_{L_{2}}(M)$.

	Therefore, in both cases,
 $M\models(T_{2,>m})_{L_{2}}$.
  Thus
	$T_{2}(T_{1,>m})\models(T_{2,>m})_{L_{2}}$,
	and hence 
	$(T_{2,>m})_{L_{2}}\subseteq T_{2}(T_{1,>m})_{L_{2}}$.
\item
	Both $T_{2,\gg0}$ and $T_{2}(T_{1,\gg0})$ are unions of chains.
	Thus
	\begin{align*}
		(T_{2,\gg0})_{L_{2}}
		=\bigcup_{m\geq0}(T_{2,>m})_{L_{2}}
		\overset{\text{(c)}}{=}
		\bigcup_{m\geq0}T_{2}(T_{1,>m})_{L_{2}}
		=\Big(\bigcup_{m\geq0}T_{2}(T_{1,>m})\Big)_{L_{2}}
		=T_{2}(T_{1,\gg0})_{L_{2}},
	\end{align*}
	as required.
	\qedhere
\end{enumerate}
\end{proof}

\section{Existential theories of fields with varying characteristic}
\label{sec:final}

\noindent
In this final section we draw the main conclusions about our examples, including theories of prime fields, PAC fields, henselian valued fields, large fields, and global fields.
We wrap up by giving the proofs of theorems~\ref{thm:introQ} and \ref{thm:intro3} from the introduction.

For the rest of this paper
we fix the language
$\mathfrak{L}_{1}$ to be $\mathfrak{L}_{\mathrm{ring}}$
and
$\rho_{n}$ to be the quantifier-free $\mathfrak{L}_{\rm ring}$-sentence
\begin{align*}
	\neg\underbrace{1+\ldots+1}_{n}\dot{=}0,
\end{align*}
for each $n\in\mathbb{N}$.
Note that
$n\mapsto\rho_{n}$ is computable.

For an arch
$A=(B,\hat{B},\iota)$
we write $B=(C_{1},C_{2},\sigma)$ and $\hat{B}=(\hat{C}_{1},\hat{C}_{2},\sigma)$,
where $C_{i}=(L_{i},T_{i})$ is a subcontext of $\hat{C}_{i}=(\hat{L}_{i},T_{i})$, for $i=1,2$.
For all such arches $A$ considered in this section, 
$C_{1}$ and $\hat{C}_{1}$ are $\mathfrak{L}_{1}$-contexts,
$L_{1}$ contains the quantifier-free $\mathfrak{L}_{\mathrm{ring}}$-sentences
(thus $\rho_{n},\neg\rho_{n}\in L_{1}$ for every $n$),
$\iota(L_{1})\subseteq L_{2}$,
and $B$ satisfies \axsurj.
Thus we may use the notation from Section~\ref{section:stratify_arch},
building on Definition~\ref{def:notations}:
first with an $\mathfrak{L}_{1}$-theory $T=R$ and the sentences $(\rho_{n})_{n\in\mathbb{N}}$,
and second with $T=T_{2}(R)$ and the sentences $(\iota\rho_{n})_{n\in\mathbb{N}}$.

\begin{remark}\label{rem:TE}
For an $\mathfrak{L}_{1}$-theory $R$ of fields,
let $S$ denote either $R$ or $T_{2}(R)$.
Then
$S_{n}$ is consistent only if $n$ is prime or $0$;
and
$S_{>0}=\bigcap_{p\in\mathbb{P}}S_{p}$
and
$S_{\gg0}=\bigcup_{m\in\mathbb{N}}\bigcap_{p\geq m,p\in\mathbb{P}}S_{p}$.
We also follow the convention from Definition \ref{def:Efragments} to write $S_\exists$, $S_{n,\exists}$, etc.,
for $S_L$, $(S_n)_L$, etc.,
when $L={\rm Sent}_\exists(\mathfrak{L}_{i})$.
Moreover by
Proposition~\ref{lem:little_equivalences}(b,d)
and
Lemma~\ref{lem:TT0}
we have
$S_{\exists}=\bigcap_{p\in\mathbb{P}\cup\{0\}}S_{p,\exists}$.
\end{remark}

\begin{remark}\label{rem:ambiguity}
Consider the arch
$A=B\arch\hat{B}=\rmF_{\exists}\bridge\Hepi_{\exists}\arch\rmF\bridge\Hepi$
together with the sentences
$(\rho_{n})_{n\in\mathbb{N}}$
defined above.
Applying the extended notations of Section~\ref{section:stratify_arch}
we have the bridge $B_{0}=(C_{1,0},C_{2,0},\sigma|_{\mathrm{Mod}((\bHepi)_{0})})$
where $C_{2,0}=(\mathrm{Sent}_{\exists}(\mathfrak{L}_{\mathrm{val}}),(\bHepi)_{0})$
and $(\bHepi)_{0}=(\bHepi\cup\{\iota_{\mathbf{k}}\rho_{n}\mid n\in\mathbb{N}\})^{\vdash}$.

In Section~\ref{section:existential_theories_of_fields}
we had already defined the theory $\bHepiz$ to be the
theory of equicharacteristic zero henselian valued fields,
which is clearly axiomatized by 
$\bHepi\cup\{\iota_{\mathbf{k}}\rho_{n}\mid n\in\mathbb{N}\}$.
Thus $\bHepiz=(\bHepi)_{0}$.
Similarly we had defined the contexts
$\Hepiz=(\mathrm{Sent}(\mathfrak{L}_{\mathrm{val}}),\bHepi_{0})$
and
$\HepizE=(\mathrm{Sent}_{\exists}(\mathfrak{L}_{\mathrm{val}}),\bHepi_{0})$.
We observe that there is no ambiguity between the notation introduced
in sections~\ref{section:existential_theories_of_fields}
and
\ref{section:stratify_arch}.
In particular,
we have $B_{0}=\rmF_{0,\exists}\bridge\HepizE$.
\end{remark}

\begin{remark}\label{rem:existential}
We make use of the following basic facts:
If $\mathcal{C}_1$ and $\mathcal{C}_2$ are two classes structures
in some language $\mathfrak{L}$,
such that
for every $M\in\mathcal{C}_1$ there exists $N\in\mathcal{C}_{2}$ with
either $N\subseteq M$ or $M\preceq_{\exists}N$,
then ${\rm Th}_\exists(\mathcal{C}_1)\supseteq{\rm Th}_\exists(\mathcal{C}_2)$.
If in addition $\mathcal{C}_1\supseteq\mathcal{C}_2$, then
${\rm Th}_\exists(\mathcal{C}_1)={\rm Th}_\exists(\mathcal{C}_2)$.
In particular, if $R_{1},R_{2}$ are two $\mathfrak{L}$-theories such that
for every $M\models R_{1}$ there exists $N\models R_{2}$ with
either $N\subseteq M$ or $M\preceq_{\exists}N$,
then $R_{1,\exists}\supseteq R_{2,\exists}$,
and if in addition $R_1\subseteq R_2$,
then $R_{1,\exists}= R_{2,\exists}$.
\end{remark}

\begin{remark}\label{rem:alldifferent}
For any theory $T$ of fields
we have $T_\exists\subseteq T_{0,\exists},T_{>0,\exists}\subseteq T_{\gg0,\exists}$ (cf.~Proposition \ref{lem:pseudo}).
If $T$ has models of every characteristic, then the sentence
$1+1\neq0$ is in $T_{0,\exists}$ but not in $T_{>0,\exists}$,
hence $T_\exists\subsetneqq T_{0,\exists}$ and
$T_{>0,\exists}\subsetneqq T_{\gg0,\exists}$.
If $T$ has a model of characteristic zero in which $\mathbb{Q}$
is algebraically closed, then
$\exists x(x^2=2\vee x^2=3\vee x^2=6)$
is in $T_{>0,\exists}$ but not in $T_{0,\exists}$,
hence $T_\exists\subsetneqq T_{>0,\exists}$
and $T_{0,\exists}\subsetneqq T_{\gg0,\exists}$.
\end{remark}

\subsection{Prime fields}
Recall that $\F$ denotes the theory of fields.
Denote $\mathbb{F}_{0}=\mathbb{Q}$
and let $\mathbf{P}=\bigcap_{p\in\mathbb{P}\cup\{0\}}{\rm Th}(\mathbb{F}_p)$ be the theory of prime fields.

\begin{remark}\label{rem:Psf}
Ax proved in \cite{Ax} that
$\mathbf{P}_{>0}$ and $\mathbf{P}_{\gg0}$ are decidable \cite[Corollary 20.9.6, Theorem 20.9.7]{FJ}.
Moreover, with 
${\bf Fin}$ the theory of finite fields and
${\bf Psf}$ the theory of pseudofinite fields (i.e.~the infinite models of the theory of finite fields), 
Ax's work also shows that $\mathbf{Psf}_0={\bf Psf}_{\gg0}=\mathbf{P}_{\gg0}=\mathbf{Fin}_{\gg0}$:
Indeed, 
for every $K\models{\bf Psf}_0$
there is an ultrafilter $\mathcal{U}$ on $\mathbb{P}$
such that $K\equiv\prod_{p\in\mathbb{P}}\mathbb{F}_p/\mathcal{U}\models\mathbf{P}_{\gg0}$ \cite[Thm.~8'']{Ax}.
Moreover, 
every $K\models\mathbf{Fin}_{\gg0}$ 
is elementarily equivalent to an ultraproduct
$\prod_{p\in\mathbb{P}}K_p/\mathcal{U}$
with each $K_p$ a finite field of characteristic $p$ (Proposition~\ref{lem:pseudo}),
and if $F_p$ is any pseudofinite field in which $K_p$ is algebraically closed (which exists by \cite[Thm.~7,8]{Ax}),
then $\prod_{p\in\mathbb{P}}K_p/\mathcal{U}\equiv\prod_{p\in\mathbb{P}}F_p/\mathcal{U}\models\mathbf{Psf}_{\gg0}$ \cite[Thm.~4]{Ax}.
Thus, $\mathbf{Psf}_0\supseteq\mathbf{P}_{\gg0}
\supseteq\mathbf{Fin}_{\gg0}\supseteq\mathbf{Psf}_{\gg0}\supseteq\mathbf{Psf}_0$ (cf.~Proposition~\ref{lem:pseudo}).
Finally, by Weil's Riemann hypothesis for curves,
$\mathbf{PAC}\subseteq\mathbf{Psf}$ \cite[Cor.~20.10.5]{FJ}.
\end{remark}

\begin{proposition}\label{prop:F_and_P}
\begin{enumerate}[(a)]
\item 
$\mathbf{F}_{p,\exists}=\mathbf{P}_{p,\exists}={\rm Th}_\exists(\mathbb{F}_p)$
for all $p\in\mathbb{P}\cup\{0\}$,
\item 
$\mathbf{F}_\exists=\mathbf{P}_\exists$,
\item 
$\mathbf{F}_{>0,\exists}=\mathbf{P}_{>0,\exists}$,
\item 
$\mathbf{F}_{\gg0,\exists}=\mathbf{P}_{\gg0,\exists}$.
\end{enumerate}
\end{proposition}

\begin{proof}
For $p\in\mathbb{P}\cup\{0\}$,
since
$\mathbf{F}_{p}\subseteq\mathbf{P}_p\subseteq{\rm Th}(\mathbb{F}_p)$ 
and every $K\models\F_p$ is a field of characteristic $p$, hence has prime field isomorphic to $\mathbb{F}_p$,
(a) follows from Remark~\ref{rem:existential}.
This 
implies (b) by Remark \ref{rem:TE},
and (c) and (d) follow immediately from (b)
and Proposition \ref{lem:little_equivalences}.
\end{proof}

\begin{corollary}\label{cor:P_decidable}
\begin{enumerate}[$(a)$]
    \item $\F_{0,\exists}={\rm Th}_\exists(\mathbb{Q})$, and $\F_{p,\exists}$ is decidable for every $p\in\mathbb{P}$.
    \item $\mathbf{F}_{\exists}\meq\mathbf{F}_{0,\exists}$,
    and $\mathbf{F}_{\exists_1}$ is decidable.
    \item $\mathbf{F}_{>0,\exists}$ is decidable.
    \item $\mathbf{F}_{\gg0,\exists}$ is decidable.
\end{enumerate}
\end{corollary}

\begin{proof}
Part (a) follows from Proposition \ref{prop:F_and_P}(a) since ${\rm Th}(\mathbb{F}_p)$ is decidable.
Parts (c) and (d) follow from the 
decidability of $\mathbf{P}_{>0}$ and $\mathbf{P}_{\gg0}$
via Proposition \ref{prop:F_and_P}(c,d),
and $\mathbf{F}_{\exists}\meq\mathbf{F}_{0,\exists}$ is then a consequence of Proposition~\ref{prop:algorithms}(b),
applied to $T=\mathbf{F}$
and $L=\mathrm{Sent}_{\exists}(\mathfrak{L}_{\mathrm{ring}})$.
Similarly, applying 
Proposition~\ref{prop:algorithms}(b) to
$T=\mathbf{F}$
and $L=\mathrm{Sent}_{\exists_1}(\mathfrak{L}_{\mathrm{ring}})$
shows
$\F_{\exists_1}\meq\F_{0,\exists_1}$,
and $\F_{0,\exists_1}={\rm Th}_{\exists_1}(\mathbb{Q})$ by Proposition \ref{prop:F_and_P}(a),
which is decidable,
see e.g.~\cite[Lemma 19.1.3]{FJ}.
\end{proof}

\subsection{PAC fields}
\label{sec:PAC2}
We easily reobtain the decidability of
$\PAC_\exists$ and $\PAC_{p,\exists}$
proven in \cite[Section 21.4]{FJ},
and in addition get the decidability of 
$\PAC_{>0,\exists}$,
and $\PAC_{\gg0,\exists}$
(which are all distinct, see Remark \ref{rem:alldifferent}).

\begin{corollary}
\label{prp:PAC}
\begin{enumerate}[$(a)$]
\item
$\PAC_{p,\exists}$ is decidable
for all $p\in\mathbb{P}\cup\{0\}$.
\item
$\PAC_{\exists}$,
$\PAC_{>0,\exists}$,
and $\PAC_{\gg0,\exists}$ are decidable.
\end{enumerate}
\end{corollary}

\begin{proof}
For (a),
by Corollary~\ref{cor:hash3}
we have
$\PAC_{p,\exists}\meq\F_{p,\exists_{1}}=\mathrm{Th}_{\exists_{1}}(\mathbb{F}_{p})$,
for each $p\in\mathbb{P}\cup\{0\}$,
which is decidable.

For (b),
by Corollary~\ref{cor:PAC}(aII) we have
$\PAC_{>0,\exists}\meq\PAC_{>0,\exists_{1}}$
and
$\PAC_{\gg0,\exists}\meq\PAC_{\gg0,\exists_{1}}$.
Considering the arch
$B\arch\hat{B}=\rmF_{\exists_{1}}\bridge\rmPAC_{\exists_1}\arch\rmF_{\exists_{1}}\bridge\rmPAC$,
we note that $B$ satisfies
\axmono\ trivially,
and $\hat{B}$ satisfies \axsurj by Corollary \ref{lem:PAC}(b).
Thus
$\PAC_{>0,\exists_{1}}=\PAC(\F_{>0})_{\exists_{1}}$
and
$\PAC_{\gg0,\exists_{1}}=\PAC(\F_{\gg0})_{\exists_{1}}$
by Proposition~\ref{prp:stratify}(c,d).
By Proposition~\ref{prp:intersection},
$\PAC(\F_{>0})_{\exists_{1}}=\PAC(\F_{>0,\exists_{1}})_{\exists_{1}}$,
and
by Corollary~\ref{cor:PAC}(bII),
we have
$\PAC(\F_{>0,\exists_{1}})_{\exists_{1}}\meq\F_{>0,\exists_{1}}$
and
$\PAC(\F_{\gg0,\exists_{1}})_{\exists_{1}}\meq\F_{\gg0,\exists_{1}}$,
both of which are decidable.
Finally
$\PAC_{\exists}\meq\F_{\exists_{1}}$,
by Corollary~\ref{cor:hash3},
which is decidable by Corollary \ref{cor:P_decidable}(b).
\end{proof}

\subsection{Henselian valued fields}
We use the above notational conventions,
applied to the arches
$\rmF_{\exists}.\Hen_{\exists}\arch\rmF.\Hen$
and
$\rmF_{\exists}.\Hepi_{\exists}\arch\rmF.\Hepi$,
together with $\mathfrak{L}_{\mathrm{ring}}$-theories
$R=\F,\mathbf{P},\mathrm{Th}(\mathbb{F}_{p})$, etc.

\begin{proposition}\label{cor:F_and_P_combined}
\quad
\begin{enumerate}[$(1)$]
\item
\begin{enumerate}[$(a)$]
\item
$\bHen_{p,\exists}=\bHen(\mathbf{F}_{p})_{\exists}=\bHen(\mathbf{P}_{p})_{\exists}=\bHen({\rm Th}(\mathbb{F}_p))_\exists=\mathrm{Th}_{\exists}(\mathbb{F}_{p}(\!(t)\!),v_{t})$
for all $p\in\mathbb{P}\cup\{0\}$,
\item
$\bHen_\exists=\bHen(\mathbf{F})_{\exists}=\bHen(\mathbf{P})_\exists$,
\item
$\bHen_{>0,\exists}=\bHen(\mathbf{F}_{>0})_{\exists}=\bHen(\mathbf{P}_{>0})_{\exists}$,
\item
$\bHen_{\gg0,\exists}=\bHen(\mathbf{F}_{\gg0})_{\exists}=\bHen(\mathbf{P}_{\gg0})_{\exists}$
\end{enumerate}
\item
\begin{enumerate}[$(a)$]
\item
$\bHepi_{p,\exists}=\bHepi(\mathbf{F}_{p})_{\exists}=\bHepi(\mathbf{P}_{p})_{\exists}=\bHepi({\rm Th}(\mathbb{F}_p))_\exists=\mathrm{Th}_{\exists}(\mathbb{F}_{p}(\!(t)\!),v_{t},t)$
for all $p\in\mathbb{P}\cup\{0\}$,
where for the last equality in case $p\in\mathbb{P}$ we require \Rfour,
\item
$\bHepi_\exists=\bHepi(\mathbf{F})_{\exists}=\bHepi(\mathbf{P})_\exists$,
\item
$\bHepi_{>0,\exists}=\bHepi(\mathbf{F}_{>0})_{\exists}=\bHepi(\mathbf{P}_{>0})_{\exists}$,
\item
$\bHepi_{\gg0,\exists}=\bHepi(\mathbf{F}_{\gg0})_{\exists}=\bHepi(\mathbf{P}_{\gg0})_{\exists}$
\end{enumerate}
\end{enumerate}
\end{proposition}

\begin{proof}
For (1) we consider the arch
$B\arch\hat{B}=\rmF_{\exists}.\Hen_{\exists}\arch\rmF.\Hen$,
and for (2) we consider the arch
$B\arch\hat{B}=\rmF_{\exists}\bridge\Hepi_{\exists}\arch\rmF\bridge\Hepi$,
just as in Lemma~\ref{lem:Hen_sat_axioms}(a,c).
Henceforth we can argue simultaneously in both cases.
The bridge
$B_{0}$
satisfies \axmono\
by
Lemma~\ref{lem:Hen_sat_axioms}(a) in case (1)
and
Lemma~\ref{lem:Hen_sat_axioms}(b) in case (2)
(note for the latter that $B_{0}=\rmF_{0,\exists}/\HepizE$,
as discussed in Remark~\ref{rem:ambiguity}).
Thus we may apply Proposition~\ref{prp:stratify}
to get the first equality in each subcase.
Since $B\sqsubseteq\hat{B}$ satisfies \axwmono,
by
Lemma~\ref{lem:wmono} and Lemma~\ref{lem:Hen_sat_axioms}(a)
in case (1),
and
Lemma~\ref{lem:Hen_sat_weak_axioms}
in case (2),
we may apply 
Lemma~\ref{lem:incomplete_monotonicity}
and
Proposition~\ref{prop:F_and_P}
to deduce the remaining equalities,
except for the final equality in (1a) (respectively, (2a)) for which we apply Corollary~\ref{cor:henselian}(aIII) (respectively, Corollary~\ref{cor:henselian}(bIII,cIII)).
\end{proof}

\begin{remark}\label{rem:wm_not_R4}
Alternatively, in case (1) of Proposition
\ref{cor:F_and_P_combined} we may argue via Proposition~\ref{prp:intersection}:
for the arch
$\rmF_{\exists}\bridge\Hen_{\exists}\arch\rmF\bridge\Hen$,
the hypotheses of the lemma hold by
Lemma \ref{lem:Hen_sat_axioms}(a).
Then
by Proposition~\ref{prop:F_and_P}(a),
for every $p\in\mathbb{P}\cup\{0\}$,
we have
$\bHen(\mathbf{F}_{p,\exists})=\bHen(\mathbf{P}_{p,\exists})=\bHen(\mathrm{Th}_{\exists}(\mathbb{F}_{p}))$.
By three applications of
Proposition~\ref{prp:intersection},
we have
$\bHen(\mathbf{F}_{p})_{\exists}=\bHen(\mathbf{P}_{p})_{\exists}=\bHen(\mathrm{Th}(\mathbb{F}_{p}))_{\exists}$.
Moreover
$(\bHen_{p})^{\vdash}=\bHen(\mathbf{F}_p)^{\vdash}$
by Proposition~\ref{prp:stratify}(b)),
and so in particular we have
$\bHen_{p,\exists}=\bHen(\mathbf{F}_p)_{\exists}$,
which proves (1a).
From this (1b) follows,
as
\begin{align*}
\bHen_{\exists}
=
\bigcap_{p\in\mathbb{P}\cup\{0\}}\bHen_{p,\exists}
=
\bigcap_{p\in\mathbb{P}\cup\{0\}}\bHen(\mathbf{P}_{p})_{\exists}
=
\bHen(\bigcap_{p\in\mathbb{P}\cup\{0\}}\mathbf{P}_{p})_{\exists}
=
\bHen(\mathbf{P})_{\exists},
\end{align*}
where the first and fourth equalities hold by Remark~\ref{rem:TE}
and the third equality holds by
by Proposition~\ref{prp:intersection}.
Similarly,
Proposition~\ref{prp:intersection} gives that
$$
 \bHen_{>0,\exists}=\bigcap_{p>0}\bHen_{p,\exists}
=\bigcap_{p>0}\bHen(\mathbf{P}_{p})_\exists
=\bHen(\bigcap_{p>0}\mathbf{P}_{p})_\exists
=\bHen(\mathbf{P}_{>0})_{\exists}
$$
and
$$
 \bHen_{\gg0,\exists}
=\bigcup_{\ell>0}\bigcap_{p>\ell}\bHen_{p,\exists}
=\bigcup_{\ell>0}\bHen(\bigcap_{p>\ell}\mathbf{P}_{p})_\exists
=\bHen(\bigcup_{\ell>0}\bigcap_{p>\ell}\mathbf{P}_{p})_\exists
=\bHen(\mathbf{P}_{\gg0})_{\exists},
$$
which proves (1c,d).

For case (2),
we consider the arch
$B\arch\hat{B}=\rmF_{\exists}\bridge\Hepi_{\exists}\arch\rmF\bridge\Hepi$.
Then a similar argument is possible provided that we suppose \Rfour.
Under this hypothesis, $B$ satisfies \axmono\
(rather than just \axwmono).
Thus Proposition~\ref{prp:intersection} applies and the above proof goes through {\em mutatis mutandis}.
The advantage of the proof given above is
that it does not rely on \Rfour.
\end{remark}

\begin{corollary}\label{cor:Hen_decidable}
\begin{enumerate}[$(1)$]
\item  
 \begin{enumerate}[$(a)$]
    \item $\bHen_{0,\exists}\meq{\rm Th}_\exists(\mathbb{Q})$,
    and $\bHen_{p,\exists}$ is decidable for every $p\in\mathbb{P}$.
    \item $\bHen_{\exists}\meq\bHen_{0,\exists}$.    
    \item $\bHen_{>0,\exists}$ is decidable.
	\item $\bHen_{\gg0,\exists}$ is decidable.
  \end{enumerate}
\item 
 \begin{enumerate}[$(a)$]
    \item $\bHepi_{0,\exists}\meq\mathrm{Th}_{\exists}(\mathbb{Q})$, 
    and if \Rfour\ holds then $\bHepi_{p,\exists}$ is decidable for every
    $p\in\mathbb{P}$.
    \item If \Rfour\ holds, then $\bHepi_{\exists}\meq\bHepi_{0,\exists}$. 
    \item If \Rfour\ holds, then $\bHepi_{>0,\exists}$ is decidable.
    \item $\bHepi_{\gg0,\exists}$ is decidable.
 \end{enumerate}
\end{enumerate}
\end{corollary}

\begin{proof}
Part (1a) (respectively (2a)) for $p=0$ follows from Proposition \ref{prop:F_and_P}(a) and 
Corollary~\ref{cor:henselian}(aII)
(respectively 
Corollary~\ref{cor:henselian}(bII)) applied to $R=\emptyset$,
and for $p\in\mathbb{P}$ follows from the decidability of $\mathbf{P}_{p}$
via Proposition
~\ref{cor:F_and_P_combined}(1a)
and Corollary~\ref{cor:henselian}(aII)
(resp.~Proposition~\ref{cor:F_and_P_combined}(2a)
and Corollary~\ref{cor:henselian}(cII)).
Similarly, (1c,d) (resp.~(2c,d) assuming \Rfour) follow from  
Corollary \ref{cor:P_decidable}(c,d)
via Proposition~\ref{cor:F_and_P_combined}(1c,d)
and Corollary~\ref{cor:henselian}(aII)
(resp.~Proposition~\ref{cor:F_and_P_combined}(2c,d)
and Corollary~\ref{cor:henselian}(cII)).
To obtain (2d) as written (i.e.~without assuming \Rfour) we argue as follows:
by Proposition~\ref{cor:F_and_P_combined}(2d),
we have 
$\bHepi_{\gg0,\exists}=\bHepi(\F_{\gg0})_{\exists}$,
and by Corollary~\ref{cor:henselian}(bII)
$\bHepiz(\F_{\gg0})_{\exists}\meq\F_{\gg0,\exists}$, which is decidable by 
Corollary \ref{cor:P_decidable}(d).
This proves (2d) since 
$\bHepi(\F_{\gg0})=\bHepiz(\F_{\gg0})$.
Furthermore (1c) (resp.~(2c)) is a consequence of Proposition~\ref{prop:algorithms}(b), applied to $T=\bHen$ (resp.~$T=\bHepi$) and
$L=\mathrm{Sent}_{\exists}(\mathfrak{L}_{\mathrm{val}})$
(resp.~$L=\mathrm{Sent}_{\exists}(\mathfrak{L}_{\mathrm{val}}(\varpi))$),
or can alternatively be deduced from 
Corollary \ref{cor:P_decidable}(c).
\end{proof}

\subsection{Large fields}
Let $\mathbf{L}$ denote the $\mathfrak{L}_{\mathrm{ring}}$-theory of large fields (cf.~Remark \ref{rem:R4}).
As every 
 PAC field is large,
 and every
 nontrivially valued henselian field is large \cite[Example 5.6.2]{Jarden},
$\mathbf{L}\subseteq\PAC$ and
$\mathbf{L}\subseteq\bHen\cap{\rm Sent}(\mathfrak{L}_{\rm ring})$.

\begin{proposition}\label{prp:large}
\begin{enumerate}[$(a)$]
\item
$\mathbf{L}_{p,\exists}=
\bHen_{p,\exists}\cap\mathrm{Sent}(\mathfrak{L}_{\mathrm{ring}})
={\rm Th}_\exists(\mathbb{F}_p(\!(t)\!))$ for all $p\in\mathbb{P}\cup\{0\}$.
\item
$\mathbf{L}_{\exists}=\bHen_{\exists}\cap\mathrm{Sent}(\mathfrak{L}_{\mathrm{ring}})$.
\item
$\mathbf{L}_{>0,\exists}=\bHen_{>0,\exists}\cap\mathrm{Sent}(\mathfrak{L}_{\mathrm{ring}})$.
\item
$\mathbf{L}_{\gg0,\exists}=\bHen_{\gg0,\exists}\cap\mathrm{Sent}(\mathfrak{L}_{\mathrm{ring}})$.
\end{enumerate}
\end{proposition}

\begin{proof}
By Proposition~\ref{cor:F_and_P_combined}(1a), $\bHen_{p,\exists}\cap\mathrm{Sent}(\mathfrak{L}_{\mathrm{ring}})=\mathrm{Th}_{\exists}(\mathbb{F}_{p}(\!(t)\!))$.
As  every nontrivially valued henselian field is large,
we have 
$\mathbf{L}\subseteq
\bHen$, in particular
$\mathbf{L}_{p,\exists}\subseteq
\bHen_{p,\exists}\cap\mathrm{Sent}(\mathfrak{L}_{\mathrm{ring}})$.
Conversely, if $K\models\mathbf{L}_p$, then $K\prec_\exists K(\!(t)\!)\models \bHen_{p,\exists}\cap\mathrm{Sent}(\mathfrak{L}_{\mathrm{ring}})$,
hence $K\models \bHen_{p,\exists}\cap\mathrm{Sent}(\mathfrak{L}_{\mathrm{ring}})$,
which proves (a).
Then (b) follows via Remark~\ref{rem:TE},
and (c) and (d) follow immediately from (a)
and Proposition \ref{lem:little_equivalences}.
\end{proof}

\begin{corollary}\label{cor:large}
\begin{enumerate}[$(a)$]
\item
$\mathbf{L}_{0,\exists}\meq{\rm Th}_\exists(\mathbb{Q})$,
and $\mathbf{L}_{p,\exists}$ is decidable for every $p\in\mathbb{P}$.
\item
$\mathbf{L}_\exists\meq\mathbf{L}_{0,\exists}$
\item
$\mathbf{L}_{>0,\exists}$ is decidable. 
\item $\mathbf{L}_{\gg0,\exists}$ is decidable.
\end{enumerate}
\end{corollary}

\begin{proof}
The decidability of
$\mathbf{L}_{>0,\exists}$ and
$\mathbf{L}_{\gg0,\exists}$
follow from
Corollary \ref{cor:Hen_decidable}(1c,d)
and
Proposition \ref{prp:large}(c,d).
For $p\in\mathbb{P}\cup\{0\}$,
$\mathbf{L}_{p,\exists}=\bHen_{p,\exists}\cap\mathrm{Sent}(\mathfrak{L}_{\mathrm{ring}})$
by
Proposition \ref{prp:large}(a),
and by
Corollary \ref{cor:Hen_decidable}(1a)
this is decidable when $p\in\mathbb{P}$.
For the case $p=0$,
$\mathbf{L}_{\exists}\meq\mathbf{L}_{0,\exists}$
is a consequence of Proposition~\ref{prop:algorithms}(b),
applied to $T=\mathbf{L}$
and $L=\mathrm{Sent}_{\exists}(\mathfrak{L}_{\mathrm{ring}})$,
and 
$\mathbf{L}_{0,\exists}=\bHen(\mathrm{Th}(\mathbb{Q}))_{\exists}\cap\mathrm{Sent}(\mathfrak{L}_{\mathrm{ring}})$
by Proposition \ref{prp:large}(a).
The latter
is many-one equivalent to $\bHen(\mathrm{Th}(\mathbb{Q}))_{\exists}$
since 
by \cite[Corollary 6.2 and Corollary 6.18]{AF17}
there are two $\mathfrak{L}_{\mathrm{ring}}$-formulas
--- one existential and one universal --- 
which each define the valuation ring
in every model of $\bHen(\mathrm{Th}(\mathbb{Q}))$.
Finally,
$\bHen(\mathrm{Th}(\mathbb{Q}))_{\exists}\meq\mathrm{Th}_{\exists}(\mathbb{Q})$
by Corollary~\ref{cor:F_and_P_combined}(1a) and Corollary~\ref{cor:Hen_decidable}(1a).
\end{proof}

\begin{remark}
We proved both 
$\bHen_{0,\exists}\Teq\bHen_\exists$
and 
$\bHen_{0,\exists}\cap\mathrm{Sent}(\mathfrak{L}_{\mathrm{ring}})\Teq\bHen_\exists\cap\mathrm{Sent}(\mathfrak{L}_{\mathrm{ring}})$
but neither statement seems to imply the other immediately,
as there is no uniform existential (or universal) definition of the valuation ring in models of $\bHen$ (or even $\bHen\cup\iota_{\mathbf{k}}\mathbf{P}$).
In particular, we obtain the somewhat surprising Turing equivalence
$\bHen_\exists\Teq\bHen_\exists\cap\mathrm{Sent}(\mathfrak{L}_{\mathrm{ring}})$.
\end{remark}

\begin{remark}
The statement $\mathbf{L}_{0,\exists}={\rm Th}_\exists(\mathbb{Q}(\!(t)\!))$,
which is part of Proposition \ref{prp:large}(a) was also proven by Sander in
\cite[Proposition 2.25]{Sander}.
Sander continues to remark (with references but without detailed proof) that
${\rm Th}_\exists(\mathbb{Q}(\!(t)\!))\Tred{\rm Th}_\exists(\mathbb{Q})$ follows from work of Weispfenning,
and he asks in \cite[Problem 2.26]{Sander}
whether $\mathbf{L}_{0,\exists}$ is decidable.
The statement $\mathbf{L}_{0,\exists}\meq{\rm Th}_\exists(\mathbb{Q})$ (Corollary \ref{cor:large}(a)) shows that this question is in fact equivalent to a famous and still open problem.
\end{remark}

\subsection{Proof of Theorem~\ref{thm:introQ}}\label{proof:1.2}
By
Corollary~\ref{cor:F_and_P_combined}(1a,2a),
we have
$\mathrm{Th}_{\exists}(\mathbb{Q}(\!(t)\!),v_{t})=\bHen(\mathrm{Th}(\mathbb{Q}))_{\exists}$
and
$\mathrm{Th}_{\exists}(\mathbb{Q}(\!(t)\!),v_{t},t)=\bHepi(\mathrm{Th}(\mathbb{Q}))_{\exists}$.
By Corollary \ref{cor:Hen_decidable}(1a,2a)
these theories are many-one equivalent to
$\mathrm{Th}_{\exists}(\mathbb{Q})$.
Thus
(a),
(c),
and
(d)
are many-one equivalent.
By Proposition~\ref{prp:large}(a),
(b) and (e)
are many-one equivalent.
By
Corollary~\ref{cor:P_decidable}(a,b),
(a)
and
(g)
are many-one equivalent.
Finally by
Corollary~\ref{cor:large}(a,b),
(a),
(e),
and
(f)
are many-one equivalent.

\label{sec:global}
\subsection{Global fields}
Let 
$\mathbf{F}^\infty$ be the theory of infinite fields and
$\mathbf{G}$ the theory of global fields
(i.e.~finitely generated fields
of characteristic zero and transcendence degree zero, or positive characteristic and transcendence degree one).

\begin{proposition}\label{lem:G_and_F}
We have $\mathbf{G}_\exists=(\mathbf{F}^\infty)_\exists$.
Consequently,
$\mathbf{G}_{\gg0,\exists}=(\mathbf{F}^\infty)_{\gg0,\exists}$,
$\mathbf{G}_{p,\exists}=(\mathbf{F}^\infty)_{p,\exists}={\rm Th}_\exists(\mathbb{F}_p(t))$
for every $p\in\mathbb{P}$,
and $\mathbf{G}_{0,\exists}=(\mathbf{F}^\infty)_{0,\exists}=\mathbf{F}_{0,\exists}={\rm Th}_\exists(\mathbb{Q})$.
\end{proposition}

\begin{proof}
Every  global field is infinite.
Conversely, let $K$ be an infinite field of characteristic $p\geq 0$.
If $p=0$, then the global field $\mathbb{Q}$ can be embedded into $K$, and if $p>0$
then the global field $\mathbb{F}_p(t)$ 
can be embedded into any uncountable elementary extension $K\prec K^*$.
It follows that $\mathbf{G}_\exists=(\mathbf{F}^\infty)_\exists$
(Remark~\ref{rem:existential}).
Using Proposition~\ref{lem:little_equivalences}, we conclude that
$\mathbf{G}_{p,\exists}=(\mathbf{G}_{\exists})_{p,\exists}=((\mathbf{F}^\infty)_\exists)_{p,\exists}=(\mathbf{F}^\infty)_{p,\exists}$ for every $p\in\mathbb{P}\cup\{0\}$,
similarly 
$\mathbf{G}_{\gg0,\exists}=(\mathbf{F}^\infty)_{\gg0,\exists}$.
As $(\F^\infty)_0=\F_0$, we conclude that 
$(\mathbf{F}^\infty)_{0,\exists}=\mathbf{F}_{0,\exists}$,
and
$\mathbf{F}_{0,\exists}={\rm Th}_\exists(\mathbb{Q})$ 
follows from Proposition \ref{prop:F_and_P}(a).
Similarly, $\mathbb{F}_p(t)$ is infinite,
and if $K$ is any infinite field of characteristic $p>0$, then $\mathbb{F}_p(t)$
can be embedded into any proper
elementary extension $K\prec K^*$.
\end{proof}

\begin{lemma}\label{prop:G_and_F}
$\mathbf{F}_{\gg0}=
(\mathbf{F}^\infty)_{\gg0}$
\end{lemma}

\begin{proof}
Since $\mathbf{F}=\mathbf{F}^\infty\cap\mathbf{Fin}$,
also
$\mathbf{F}_{\gg0}=(\mathbf{F}^\infty)_{\gg0}\cap\mathbf{Fin}_{\gg0}$.
So as 
$\mathbf{Fin}_{\gg0}=\mathbf{Psf}_{\gg0}$ 
(Remark~\ref{rem:Psf})
and $\mathbf{Psf}\supseteq\mathbf{F}^\infty$,
in particular $\mathbf{Psf}_{\gg0}\supseteq(\mathbf{F}^\infty)_{\gg0}$,
we get that
$\mathbf{F}_{\gg0}=(\mathbf{F}^\infty)_{\gg0}\cap\mathbf{Fin}_{\gg0}=(\mathbf{F}^\infty)_{\gg0}$.
\end{proof}

\begin{remark}
In particular, using Proposition \ref{prop:F_and_P}(d), Proposition \ref{lem:G_and_F}, 
and Lemma \ref{prop:G_and_F}
we obtain
$\mathbf{P}_{\gg0,\exists}=
\mathbf{F}_{\gg0,\exists}=
(\mathbf{F}^\infty)_{\gg0,\exists}=
\mathbf{G}_{\gg0,\exists}$.
We can also show that all of these theories are equal to
$\mathbf{L}_{\gg0,\exists}$ and to $\mathbf{PAC}_{\gg0,\exists}$:
As $\mathbb{F}_p\subseteq\mathbb{F}_p(t)\subseteq\mathbb{F}_p(\!(t)\!)$, 
from 
Propositions \ref{prop:F_and_P}(a),
\ref{lem:G_and_F} and
\ref{prp:large}(a)
we get
$\mathbf{F}_{p,\exists}\subseteq\mathbf{G}_{p,\exists}\subseteq\bHen_{p,\exists}\cap{\rm Sent}(\mathfrak{L}_{\rm ring})$, and therefore
$$
 \mathbf{P}_{\gg0,\exists}=\mathbf{F}_{\gg0,\exists}\subseteq\mathbf{G}_{\gg0,\exists}\subseteq\bHen_{\gg0,\exists}\cap{\rm Sent}(\mathfrak{L}_{\rm ring})=\mathbf{L}_{\gg0,\exists}\subseteq\mathbf{PAC}_{\gg0,\exists}
$$ 
using Proposition \ref{prp:large}(d).
By Remark \ref{rem:Psf}
we have $\mathbf{PAC}_{\gg0}\subseteq\mathbf{Psf}_{\gg0}=\mathbf{Psf}_0=\mathbf{P}_{\gg0}$,
and thus $\mathbf{PAC}_{\gg0,\exists}=\mathbf{P}_{\gg0,\exists}$ (although $\mathbf{PAC}_{\gg0}\neq\mathbf{P}_{\gg0}$).
Note that $\mathbf{F}_{\gg0,\exists}\supsetneqq\mathbf{F}_{0,\exists}$, $\mathbf{PAC}_{\gg0,\exists}\supsetneqq\mathbf{PAC}_{0,\exists}$ and
$\bHen_{\gg0,\exists}\supsetneqq\bHen_{0,\exists}$,
as each of $\mathbf{F}$, $\mathbf{PAC}$ and $\bHen$
has models of characteristic zero in which $\mathbb{Q}$
is algebraically closed
(cf.~Remark \ref{rem:alldifferent}).
\end{remark}

\begin{corollary}\label{cor:last}
$\mathbf{G}_{\mathbb{N},\exists} \Teq \mathbf{G}_{>0,\exists}$ and
$\mathbf{G}_{\mathbb{N}\cup\{0\},\exists} \Teq \mathbf{G}_{\exists}$.   
\end{corollary}

\begin{proof}
It follows from Lemma \ref{prop:G_and_F} that 
$\mathbf{F}_{\gg0,\exists}=
(\mathbf{F}^\infty)_{\gg0,\exists}$.
Therefore, using
Corollary \ref{cor:P_decidable}(c,d), Proposition \ref{prop:algorithms}(d)
for $T=\mathbf{F}^\infty$, $T'=\mathbf{F}$
and $L={\rm Sent}_\exists(\mathfrak{L}_{\rm ring})$
gives that
$(\mathbf{F}^\infty)_{>0,\exists}\Teq(\mathbf{F}^\infty)_{\mathbb{N},\exists}$
and 
$(\mathbf{F}^\infty)_{\exists}\Teq(\mathbf{F}^\infty)_{\mathbb{N}\cup\{0\},\exists}$.
As $(\mathbf{F}^\infty)_\exists=\mathbf{G}_\exists$ (Proposition \ref{lem:G_and_F}),
the claim follows via Proposition~\ref{lem:little_equivalences}.
\end{proof}

\subsection{Proof of Theorem~\ref{thm:intro3}}\label{proof:1.3}
The claim is that
$(\F^\infty)_\exists$ is decidable if and only if
both 
${\rm Th}_\exists(\mathbb{Q})$ and
$$
 U:=\{ (\varphi,p) : p\in\mathbb{P},\varphi\in{\rm Th}_\exists(\mathbb{F}_p(t)) \}
$$
are decidable.
By Proposition~\ref{lem:G_and_F},
$(\mathbf{F}^\infty)_\exists=\mathbf{G}_\exists$,
${\rm Th}_\exists(\mathbb{Q})=\mathbf{G}_{0,\exists}$ and
$U=\mathbf{G}_{\mathbb{N},\exists}\cap({\rm Sent}_\exists(\mathfrak{L}_{\rm ring})\times\mathbb{P})$.
By
Corollary~\ref{cor:last},
$\mathbf{G}_{\exists}$ is decidable if and only if
$\mathbf{G}_{\mathbb{N}\cup\{0\},\exists}$ is decidable,
which is the case if and only if 
both
$\mathbf{G}_{0,\exists}$
and 
$\mathbf{G}_{\mathbb{N},\exists}$
are decidable.


\section*{Acknowledgements}

\noindent
The authors warmly thank Philip Dittmann for
a very careful reading of preliminary versions and myriad helpful suggestions
that improved this work in many ways.
They also would like to thank the anonymous referee for concrete recommendations that made the paper more readable.

Part of this is based upon work supported by the National Science Foundation under Grant No. DMS-1928930 while the authors participated in a program hosted by the Mathematical Sciences Research Institute in Berkeley, California, during Summer 2022.
S.~A.\ was supported by GeoMod AAPG2019 (ANR-DFG).
S.~A.\ and A.~F.\ were supported by the Institut Henri Poincaré.
A.~F.~was funded by the Deutsche Forschungsgemeinschaft (DFG) - 404427454.

\def\bibfont{\footnotesize}
\bibliographystyle{plain}

\end{document}